\documentclass[12pt]{amsart}
\usepackage{amssymb,enumerate}
\usepackage{amsmath}

\usepackage{tabstackengine}

\usepackage[colorlinks=true, pdfstartview=FitV, linkcolor=blue, citecolor=blue, urlcolor=blue]{hyperref}
\usepackage{bm}
\usepackage{fullpage}

\newtheorem{theorem}{Theorem}[section]
\newtheorem{lemma}[theorem]{Lemma}
\newtheorem{proposition}[theorem]{Proposition}
\newtheorem{remark}[theorem]{Remark}
\numberwithin{equation}{section}

\newcommand{\RR}{\mathbb{R}}
\newcommand{\beas}{\begin{eqnarray*}}
\newcommand{\eeas}{\end{eqnarray*}}
\newcommand{\ep}{\epsilon}

\begin{document}


\title{Stochastic Burgers' Equation on the Real Line: Regularity and Moment Estimates}

\author{Peter Lewis}

\author{David Nualart$^*$ }\thanks{$^*$David Nualart is partially supported by the NSF grant DMS1512891}

\address{%
Department of Mathematics\\
University of Kansas\\
Lawrence, KS 66045, USA}
\email{plewis85@ku.edu}
 \email{nualart@ku.edu}

\begin{abstract}
In this project we investigate the stochastic Burgers' equation with multiplicative space-time white noise on an unbounded spatial domain. We give a random field solution to this equation by defining a process via a kind of Feynman-Kac representation which solves a stochastic partial differential equation such that its Hopf-Cole transformation solves Burgers' equation. Finally, we obtain H\"older regularity and moment estimates for the solution to Burgers' equation.
\end{abstract}

\maketitle

\section{Introduction} 

We are concerned with the following (formal) version of Burgers' equation
\begin{align*}
\frac{\partial}{\partial t} u(t,x) = \frac{\partial^2}{\partial x^2} u(t,x) -\frac{1}{2} \frac{\partial}{\partial x} u(t,x)^2 + \sigma(t,x,u(t,x))\frac{\partial^2 W}{\partial t \partial x}
\end{align*}
indexed by $(t,x) \in [0,T]\times \mathbb{R}$, given a nonrandom initial condition $u_0$ and a Brownian sheet $W$. To study this equation rigorously, we understand the above in its \emph{mild form}; that is, as an integral equation:
\begin{equation}
\begin{aligned}  \label{eq1}
 u(t,x) & = \int_{\mathbb{R}} G_t(x-y) u_0(y) dy + \frac{1}{2} \int_0^t \int_{\mathbb{R}}\frac{\partial }{\partial y}G_{t-s}(x-y) u(s,y)^2 dy ds  \\
& + \int_0^t \int_{\mathbb{R}} G_{t-s}(x-y) \sigma_s(y) W(ds,dy),
\end{aligned}
\end{equation}
where $\sigma_t(x) \equiv \sigma(t,x,u(t,x))$ is used for shorthand, $G$ is the heat kernel $$G_t(x) = (4\pi t)^{-1/2} e^{-x^2/4t},$$ and the stochastic integral is understood in the Walsh sense. 

In \cite{burger1}, the authors investigate this equation with $\sigma \equiv 1$; that is, with \emph{additive noise}. In \cite{burger3}, multiplicative noise ($\sigma$ depending on $u$) is studied on the spatial interval $[0,1]$, rather than on $\mathbb{R}$. The authors of both of these papers construct an explicit solution by defining a process via a Feynman-Kac representation such that its Hopf-Cole transformation solves Burgers' equation. On the other hand, existence and uniqueness to a general class of semilinear stochastic partial differential equations (SPDEs) on unbounded spatial domains, which contains (\ref{eq1}), is shown by Gy\"{o}ngy and Nualart in \cite{burger2} using fixed point arguments which follow from some maximal inequalities on stochastic convolutions. However, they do not consider an explicit construction of the solution. Hence, the primary aim of this paper is to construct a solution to (\ref{eq1}), similar to what is done in \cite{burger1} and \cite{burger3}, by defining and transforming a process with a Feynman-Kac representation. We then obtain uniqueness for free from \cite{burger2}. Finally, we prove H\"older regularity and moment estimates to the solution of Burgers' equation. In the case of additive noise, such moment estimates have been studied. However, to our knowledge, H\"older regularity has not been established in other works on stochastic Burgers' equation with space-time white noise. 

The paper is organized as follows. First, we define a process, $\psi$, via a kind of Feynman-Kac representation. Then, we establish several properties of $\psi$, such as moment bounds, H\"older regularity, and differentiability. Next, we show that the Hopf-Cole transformation of $\psi$, which (formally at the moment) is
\begin{align*} u(t,x) = - 2\frac{\partial }{\partial x} \log \psi(t,x), \end{align*}
solves  (\ref{eq1}). Appealing to the uniqueness result in \cite{burger2}, our solution is unique. Lastly, we obtain H\"older regularity and an upper bound on moments of the solution to Burgers' equation using properties of the process $\psi$. 

Throughout much of the paper, we follow similar steps as in \cite{burger3}, but have to adjust almost all of the arguments to handle the challenges posed by an unbounded domain. As such, due to difficulties with integrability, many of our assumptions differ from those in \cite{burger3}, though they are consistent with \cite{burger2}.


\section{Preliminaries}

Let $W = \{W(t,x) , t \in  \RR_+, x \in \mathbb{R}\}$ be a zero-mean Gaussian random field defined on a complete probability space $(\Omega, \mathcal{F}, \mathbb{P})$, with covariance
$$
\mathbb{E}[W(s,x)W(t,y)] = (s\wedge t) (|x| \wedge |y|)\mathbf{1}_{[0,\infty)}(xy)
$$
for $s,t  \ge 0$, $x,y \in \mathbb{R}$. In other words, $W$ is a \emph{Brownian sheet} on $\mathbb{R}^2$. For any $t\ge 0$, we denote by $\mathcal{F}_t$ the $\sigma$-field generated by the random variables $\{W(s,x), s\in [0,t], x\in \mathbb{R}\}$ and the sets of probability zero.
The stochastic integral with respect to $W$ in (\ref{eq1}) is understood in the Walsh sense. For a careful treatment of this integration theory, see John Walsh's seminal work on SPDEs \cite{walsh}, for example. We use the notation $\mathbb{E}(\cdot)$ to represent expectation with respect to $W$, and denote its corresponding norm by $\| \cdot \|_p = \mathbb{E}(|\cdot|^p)^{1/p}$.   We will make use of the convention $\mathbf{1}_{[b,a]} = -\mathbf{1}_{[a,b]}$, whenever $b>a$.

Throughout the paper we assume the following conditions:

\medskip
\noindent
(A1) The initial condition $u_0$ is a deterministic, continuous, and bounded function such that $u_0 \in L^2(\mathbb{R})\cap L^1(\mathbb{R})$.

\noindent
(A2) $\sigma:  \RR_+ \times \mathbb{R} ^2\rightarrow \mathbb{R}$ is a Borel function satisfying the following Lipschitz and growth properties
\begin{eqnarray}
|\sigma(t,x,r) - \sigma(t,x,v)|  &\le&  L |r-v|  \\    \label{eq2}
|\sigma(t,x,r)|  &\le&  f(x) \label{eq3}
\end{eqnarray}
for all $t \ge 0$, $x,r,v \in \mathbb{R}$ and  for some constant $L>0$ and some non-negative  function  $f \in L^2(\mathbb{R})\cap L^q(\mathbb{R})$, where $q>2$.

\medskip
Under these conditions, it is proved by  Gy\"{o}ngy and Nualart in \cite{burger2} that there exists a unique $L^2(\mathbb{R})$-valued $\mathcal{F}_t$-adapted continuous  stochastic process $u=\{u(t), t \ge 0\}$, which satisfies the integral  equation  (\ref{eq1}). Furthermore, the process $u$  has a continuous version in $(t,x)$.

Before our discussion of the Feynman-Kac representation, we prove a technical lemma regarding regularity of the heat kernel $G_t(x)=(4\pi t)^{-1/2} e^{-x^2/4t}$ that  will be used several times in the paper.

\begin{lemma} \label{lem1}
Let $\theta_1 >0$, $\theta_2 \ge 0$ and $\beta >0$  be such that
\begin{equation}  \label{kk1}
\beta( \theta_1-\theta_2-1) <2< \beta( 3\theta_1  -\theta_2-1).
\end{equation}
Then, for any $0 < t_1 <t_2 $, we have
\[
\int_0^{t_1}  \left( \int_{\mathbb{R}} | G_{t_2-s}(x) - G_{t_1-s}(x)|^{\theta_1} |x|^{\theta_2} dx \right)^{\beta} ds 
\le C (t_2-t_1) ^{1- \beta( \theta_1-\theta_2-1)/2},
\]
for some constant $C$ depending on $\theta_1$, $\theta_2$ and $\beta$.
\end{lemma}

\begin{proof}
Set $\tau = t_2-t_1$. Making the change of variables $x=\sqrt{s} y$ and $s=\tau/\sigma$, yields
\beas
&& \int_0^{t_1}  \left( \int_{\mathbb{R}} | G_{t_2-s}(x) - G_{t_1-s}(x)|^{\theta_1} |x|^{\theta_2} dx \right)^{\beta} ds  \\
&& \quad =
\int_0^{t_1}  \left( \int_{\mathbb{R}} \left|  \frac 1{ \sqrt{4\pi (\tau+s)}} e^{-\frac {x^2}{ 4(\tau+s)}} -  \frac 1{\sqrt{4\pi s}} e^{-\frac {x^2}{ 4s}}  \right|^{\theta_1} |x|^{\theta_2} dx \right)^{\beta} ds \\
&& \quad \le C  (t_2-t_1) ^{1- \beta( \theta_1-\theta_2-1)/2}
\int_0^\infty \sigma^{ -2+ \beta( \theta_1 -\theta_2-1)/2}  \\
&&\qquad \qquad \times
  \left( \int_{\mathbb{R}} \left|  \frac 1{ \sqrt{\sigma+1}} e^{-\frac {y^2}{ 4(\sigma+1)}} -  e^{-\frac {y^2}{ 4}}  \right|^{\theta_1} |y|^{\theta_2} dy \right)^{\beta} d\sigma.
\eeas
Then,  condition (\ref{kk1}) implies that the above integral in $d\sigma$ is finite, and we get the desired estimate.
\end{proof}
 
Throughout the paper we will denote by $C$ a generic constant that might depend on   $\sigma$, $f$, $u_0$, $T$ and the exponent $p$ we are considering. The value of this constant may be different from line to line. However, we will specify dependence where we feel it may be relevant.


\section{Feynman-Kac Representation}
We now define a process via a kind of Feynman-Kac formula that will be the main focus of this paper. Given $u_0$, set 
$$
 \psi_0(x):= \exp\left\{-\frac{1}{2}\int_0^x u_0(y)dy \right\}
$$
with the convention that the integral is on the interval $[x,0]$ if $x< 0$. Let  $\beta=\{\beta_s, s\in [0,t]\}$ be a backward Brownian motion (BWBM) that is independent of $W$, starting at $x \in \mathbb{R}$ at time $t$ and with variance $2(t-s)$. We use the notation $\mathbb{E}_{x,t}^{\beta}$ to denote the expectation with respect to the law of the BWBM. 
Let $u$ be the mild solution to Burgers' equation. That is, $u$ satisfies (\ref{eq1}). We will make use of  the notation $\sigma_s(y):= \sigma(s,y,u(s,y))$.  Set 
$$
M_t^{\beta}:= \int_0^t \int_{\mathbb{R}} \sigma_s(y) \mathbf{1}_{[0,\beta_s]}(y) W(ds,dy), \hspace{1mm}
$$
with the convention that the indicator function is on the interval $[\beta_s,0]$ if the BWBM is negative at time $s$. Observe that this stochastic integral is a well-defined martingale due to the square-integrability assumption (\ref{eq3}) on $\sigma$. With the above notation in mind, define the two-parameter stochastic process $\psi$ by
\begin{equation}
 \label{eq4} \psi(t,x) := \mathbb{E}_{x,t}^{\beta} \left[ \psi_0(\beta_0) e^{-\frac{1}{2} M_t^{\beta}}\right].
\end{equation}
We first establish some estimates of moments of the process $\psi$, then show that it satisfies a certain integral equation. 

 \begin{proposition}  \label{prop1}
For all $t  \ge 0$, $x \in \mathbb{R}$, and integers $p \geq 2$, we have moment estimates of the form
\begin{equation} \label{eq8}
 \|\psi(t,x)\|_p \leq  \exp\left( \frac { tp}4 \|f\|^2_{L^2(\RR)} + \frac 12 \| u_0\|_{L^1(\RR)} \right).
\end{equation}
\end{proposition}

\begin{proof}
Let $\vec{\beta} = \{\beta^i\}_{i=1}^p$ be $p$ independent  backward Brownian motions on $[0,t]$ starting at $x \in \mathbb{R}$ at time $t$, with variance $2(t-s)$. By independence and Fubini's theorem,  we have
\begin{align*}
\|\psi(t,x)\|_p^p = \mathbb{E}\big( | \psi(t,x)|^p\big) &= \mathbb{E} \left[ \prod_{i=1}^p \mathbb{E}^{\beta^i}_{x,t}\left(\psi_0(\beta_0^i)e^{-\frac{1}{2} M_t^{\beta^i}} \right) \right]\\
&= \mathbb{E} \left[  \mathbb{E}^{\vec{\beta}}_{x,t}\left(\prod_{i=1}^p\psi_0(\beta_0^i)e^{-\frac{1}{2} M_t^{\beta^i}} \right) \right]\\
&=\mathbb{E}^{\vec{\beta}}_{x,t}  \left[ \left(\prod_{i=1}^p\psi_0(\beta_0^i)\right) \mathbb{E}\left(\exp\Big\{-\frac{1}{2} \sum_{j=1}^p M_t^{\beta^j}\Big\} \right) \right].\\
\end{align*}
Now, by the multivariate It$\hat{\text{o}}$ formula,
\begin{align*}
e^{-\frac{1}{2} \sum_{j=1}^p M_t^{\beta^j}} &= 1 -\frac 12 \int_0^t \sum_{i=1}^p  e^{-\frac{1}{2} \sum_{j=1}^p M_s^{\beta^j}} dM_s^{\beta^i}+\int_0^t \sum_{i,j=1}^p \frac{1}{8} e^{-\frac{1}{2} \sum_{k=1}^p M_s^{\beta^k}}d\langle M^{\beta^i},M^{\beta^j}\rangle_s .
\end{align*}
Since the quadratic covariation of these martingales is 
\[
\langle M^{\beta^i},M^{\beta^j}\rangle_t = \int_0^t ds \int_{\mathbb{R}}dy \hspace{1mm} \sigma^2_s(y)  \mathbf{1}_{[0,\beta_s^i]}(y)\mathbf{1}_{[0,\beta_s^j]}(y) ,
\]
taking the expectation of the above It$\hat{\text{o}}$ expansion yields
\begin{align*}
\mathbb{E}\Big(e^{-\frac{1}{2} \sum_{j=1}^p M_t^{\beta^j}}\Big)&= 1 +\frac{1}{8} \int_0^t  \hspace{1mm}\mathbb{E} 
\Big(  e^{-\frac{1}{2} \sum_{k=1}^p M_s^{\beta^k}} \sum_{i,j=1}^p \mathbf{1}_{\beta_s^i \beta_s^j >0} \int_0^{|\beta^i| \wedge |\beta^j|} [\sigma^2_s(y)+ \sigma^2_s(-y)] dy \Big)  ds \\
&\leq 1+\frac{p^2}{4} \|f\|_{L^2(\mathbb{R})}^2\int_0^t \mathbb{E} \Big(e^{-\frac{1}{2} \sum_{j=1}^p M_s^{\beta^j}}\Big) ds.
\end{align*}
Recall a version of Gronwall's lemma which states that if a function $g$ satisfies 
$
\displaystyle g(t) \leq a(t) + \int_0^t b(s) g(s) ds,
$
where $a$ is non-decreasing and $b$ is non-negative, then $g$ satisfies
$
\displaystyle g(t) \leq a(t) e^{\int_0^t b(s) ds}.
$
Hence, we have
$$
\mathbb{E}\Big(e^{-\frac{1}{2} \sum_{j=1}^p M_t^{\beta^j}}\Big) \leq \exp\Big(\frac{\|f\|^2_{L^2(\mathbb{R})}}{4} t p^2\Big).
$$
Therefore,
$$
\|\psi(t,x)\|_p^p =\mathbb{E}^{\vec{\beta}}_{x,t}  \left[ \left(\prod_{i=1}^p\psi_0(\beta_0^i)\right) \mathbb{E}\left(\exp\Big\{-\frac{1}{2} \sum_{j=1}^p M_t^{\beta^j}\Big\} \right) \right] \leq a^p e^{b t p^2},
$$
where $a=e^{\frac 12\|u_0\|_{L^1(\mathbb{R})}}$ and $b= \frac{1}{4}\|f\|^2_{L^2(\mathbb{R})}$.
\end{proof}

\begin{remark}  \label{rem1} 
Using Jensen's inequality we can show, in the same way as before, that for all integers $p\ge 2$,
\begin{equation}  \label{kk7}
  \|\psi(t,x)^{-1}\|_p ^p \le  \exp\left( \frac { tp}8 \|f\|^2_{L^2(\RR)} + \frac 12 \| u_0\|_{L^1(\RR)} \right).
 \end{equation}
 In fact,
 \[
 \psi(t,x)^{-1} \le \mathbb{E} ^\beta_{x,t} \left[ \exp\left\{ \frac 12 \int_0^x u_0(y) dy +\frac 12 M_t ^\beta \right\} \right].
 \]
Proposition \ref{prop1} implies that  for any $T>0$
\begin{equation} \label{mp}
M_{p,T}: = \sup_{t\in [0,T], x\in \mathbb{R}} \| \psi(t,x)\|_p <\infty
\end{equation}
  and 
\begin{equation} \label{mp2}
 \sup_{t\in [0,T], x\in \mathbb{R}} \| \psi(t,x)^{-1}\|_p <\infty
\end{equation}
for all real numbers $p\ge 2$.
 \end{remark}


Next, we show that $\psi$ satisfies a particular integral equation.
\begin{proposition}  \label{prop3.2}
Let $\psi$ be the process defined in (\ref{eq4}) and  let $G_t(x)$ be the heat kernel as before. Then, for $t \ge0 $, $x \in \mathbb{R}$, $\psi(t,x)$ satisfies
 \begin{equation}  \label{eq5}
 \begin{aligned}
 \psi(t,x) &= \int_{\mathbb{R}} G_t(x-y) \psi_0(y) dy -\frac{1}{2} \int_0^t \int_{S} {\rm sign} (y)G_{t-s}(x-z) \psi(s,z) \sigma_s(y) dz W(ds,dy)\\
& + \frac{1}{8} \int_0^t \int_{S} G_{t-s}(x-z) \psi(s,z) \sigma^2_s(y)  dz dyds,
\end{aligned}
\end{equation}
where $$
S := \big\{(y,z) \in \mathbb{R}^2 : |z| \geq |y|, \text{ and }~ yz \geq 0 \big\}.
$$
\end{proposition}
\begin{proof} The proof of this result follows from the same arguments as  in \cite{burger3}. We  briefly explain the main idea.
First, observe that $\beta_0$ satisfies $\displaystyle \mathbb{E}_{x,t}^{\beta}\big(\psi_0(\beta_0)\big) = \int_{\mathbb{R}} G_t(x-y) \psi_0(y)dy$ since $y\mapsto G_t(x-y)$ is the   density of $\beta_0$. Now, apply It\^{o}'s formula to get
$$
e^{-\frac{1}{2} M_t^{\beta}}=1  -\frac{1}{2} \int_0^t \int_{\mathbb{R}} e^{-\frac{1}{2} M_s^{\beta}} \sigma_s(y)\mathbf{1}_{[0,\beta_s]}(y) W(ds,dy)+\frac{1}{8} \int_0^t \int_{\mathbb{R}} e^{-\frac{1}{2} M_s^{\beta}} \sigma_s(y)^2| \mathbf{1}_{[0,\beta_s]}(y)| dy ds.
$$
Multiply by $\psi_0(\beta_0)$ and take the expectation with respect to the BWBM to see that
\begin{align*}
\psi(t,x) = \int_{\mathbb{R}} G_t(x-y) \psi_0(y)dy -\frac{1}{2}\int_0^t \int_{\mathbb{R}}\sigma_s(y) \mathbb{E}_{x,t}^{\beta} \Big(\psi_0(\beta_0)e^{-\frac{1}{2}M_s^{\beta}}\mathbf{1}_{[0,\beta_s]}(y)\Big) W(ds,dy) \\+ \frac{1}{8}\int_0^t \int_{\mathbb{R}} \sigma_s(y)^2 \hspace{0.5mm} \mathbb{E}_{x,t}^{\beta} \Big(\psi_0(\beta_0)e^{-\frac{1}{2}M_s^{\beta}}|\mathbf{1}_{[0,\beta_s]}(y)|\Big) dy ds.
\end{align*}
Finally, apply the Markov property to get
\begin{align*}
\mathbb{E}_{x,t}^{\beta} \Big(\psi_0(\beta_0)e^{-\frac{1}{2}M_s^{\beta}}\mathbf{1}_{[0,\beta_s]}(y)\Big)
&= \mathbb{E}_{x,t}^{\beta}\Big[ \mathbb{E}\Big(\psi_0(\beta_0)e^{-\frac{1}{2}M_s^{\beta}}\mathbf{1}_{[0,\beta_s]}(y)\Big| \beta_r, s \leq r \leq t \Big)\Big] \\  
&=\mathbb{E}_{x,t}^{\beta}\Big[\mathbf{1}_{[0,\beta_s]}(y) \mathbb{E}\Big(\psi_0(\beta_0)e^{-\frac{1}{2}M_s^{\beta}}\Big| \beta_s \Big)\Big] \\
&= \int_0^{\infty}G_{t-s}(x-z) \mathbb{E}_{z,s}^{\beta}\Big( \psi_0(\beta_0) e^{-\frac{1}{2}M_s^{\beta}}\Big) \mathbf{1}_{[0,z]}(y)  dz \\
& \hspace{10mm}- \int_{-\infty}^0 G_{t-s}(x-z) \mathbb{E}_{z,s}^{\beta}\Big( \psi_0(\beta_0) e^{-\frac{1}{2}M_s^{\beta}}\Big) \mathbf{1}_{[z,0]}(y)  dz\\
&= 
\left\{ \begin{array}{c c }
\displaystyle \int_y^{\infty}G_{t-s}(x-z) \psi(s,z)dz & \text{if } y \geq 0 \vspace{2mm} \\
\displaystyle -\int_{-\infty}^yG_{t-s}(x-z) \psi(s,z)dz  & \text{if } y<0 .
\end{array} \right.
\end{align*}
Similarly,
\begin{align*}
\mathbb{E}_{x,t}^{\beta} \Big(\psi_0(\beta_0)e^{-\frac{1}{2}M_s^{\beta}}|\mathbf{1}_{[0,\beta_s]}(y)|\Big)
&= 
\left\{ \begin{array}{c c }
\displaystyle \int_y^{\infty}G_{t-s}(x-z) \psi(s,z)dz & \text{if } y \geq 0 \vspace{2mm} \\
\displaystyle \int_{-\infty}^yG_{t-s}(x-z) \psi(s,z)dz  & \text{if } y<0 .
\end{array} \right.
\end{align*}
Hence, we have the desired result.
\end{proof}
Next, we establish a H\"{o}lder regularity property for $\psi$.
\begin{proposition} \label{prop3.4}
For $p \geq 2$ and $T>0$, there exists some constant $C$, depending on $p$,  $T$,  $\|u_0\|_\infty$, $\|u_0\|_{L^1(\mathbb{R})}$, and $\| f\|_{L^2(\mathbb{R})}$, such that for all
 $s,t \in [0,T]$, and   $ x, y \in \mathbb{R}$,  
$$
\| \psi(t,x) - \psi(s,y)\|_p \leq C \left(|t-s|^{1/2}+ |x-y|^{1/2}\right).
$$
\end{proposition}

\begin{proof} 
First we prove the H\"older continuity in the space variable.
 Let $x_1$ and $x_2$ be such that $|x_1-x_2| =\delta $. 
Because $\|\psi(t,x)\|_p$ is uniformly bounded on $[0,T]\times \mathbb{R}$, we can assume that $\delta \le 1$. 
We have
\begin{eqnarray*}
&&
 \psi(t,x_1) -\psi(t,x_2)\\
 && = \int_{\mathbb{R}} [G_t(x_1-y)- G_t(x_2-y) ] \psi_0(y) dy \\
&&  \quad  -\frac{1}{2} \int_0^t \int_{S} {\rm sign} (y)[G_{t-s}(x_1-z)- G_{t-s}(x_2-z)] \psi(s,z) \sigma_s(y) dz W(ds,dy)\\
&&   \quad + \frac{1}{8} \int_0^t \int_{S} [G_{t-s}(x_1-z)- G_{t-s} (x_2-z)] \psi(s,z) \sigma^2_s(y)  dz dyds\\
&& =: I_1(x_1,x_2) -\frac 12 I_2(x_1,x_2)+ \frac 18I_3(x_1,x_2).
\end{eqnarray*}
We make a change of variables to get
\[
|I_1(x_1,x_2)|\leq    \int_{\mathbb{R}} G_t(u) |\psi_0(x_1-u) - \psi_0(x_2-u)| du.
\]
By  Hypothesis (A1) the function $\psi_0$ has a bounded derivative:
\[
|\psi'(x)| \le \|u_0\|_\infty e^{\frac 12\|u_0\|_{L^1(\mathbb{R})}}.
\]
Therefore, it is Lipschitz and we obtain
\begin{equation} \label{z3}
|I_1(x_1,x_2)|\leq C|x_1-x_2| =C\delta.
\end{equation}
Consider the decomposition
\[
I_2(x_1, x_2)= I_{2,+} (x_1, x_2) + I_{2,-} (x_1,x_2),
\]
where 
\[
 I_{2,+} (x_1, x_2)= \int_0^t  \int_0 ^\infty  \sigma_s(y)\int_y ^\infty[G_{t-s}(x_1-z)- G_{t-s}(x_2-z)] \psi(s,z)  dz W(ds,dy)
 \]
 and
 \[
 I_{2,-} (x_1, x_2)= -\int_0^t  \int_{-\infty}^0  \sigma_s(y)\int_{-\infty} ^y[G_{t-s}(x_1-z)- G_{t-s}(x_2-z)] \psi(s,z)  dz W(ds,dy).
 \]
Applying Burkholder's and Minkowski's inequalities, we get
\begin{eqnarray}  \notag
&&  \|  I_{2,+}(x_1,x_2 ) \|_p  \notag  \\ \notag
& & \quad  \leq  c_p\left\| \int_0^t \int_0^{\infty} \sigma^2_s(y) \Big( \int_y^{\infty} \psi(s,z) \big[G_{t-s}(x_1-z)-G_{t-s}(x_2-z)\big]dz\Big)^2 dy ds \right\|_{p/2}^{1/2}\\   
&& \quad  \leq   c_p\left(\int_0^t \int_0^{\infty}  f^2(y)  \left\|  \int_y^{\infty} \psi(s,z) \big[G_{t-s}(x_1-z)-G_{t-s}(x_2-z)\big]dz \right\| _p^2 dy ds \right)^{1/2}. \label{z2}
\end{eqnarray}
Making a change of variables we can write
\begin{eqnarray*}
&&\int_y^{\infty} \psi(s,z) \big[G_{t-s}(x_1-z)-G_{t-s}(x_2-z)\big]dz \\
&& \quad =\int_{-\infty}^{x_1-y} \psi(s,x_1-u) G_{t-s}(u)du-   \int_{-\infty}^{x_2-y} \psi(s,x_2-u) G_{t-s}(u)du\\
&& \quad =\int_{-\infty}^{x_1-y} [\psi(s,x_1-u) -\psi(x_2-u)] G_{t-s}(u)du +  \int^{x_1-y}_{x_2-y} \psi(s,x_2-u) G_{t-s}(u)du.
\end{eqnarray*}
This leads to the estimate
\begin{eqnarray*}
&& \left\|\int_y^{\infty} \psi(s,z) \big[G_{t-s}(x_1-z)-G_{t-s}(x_2-z)\big]dz\right\|_p \\
&& \quad \le  \int_{\mathbb{R}} \|\psi(s,x_1-u) -\psi(s,x_2-u) \|_pG_{t-s}(u)du +  \int^{x_1-y}_{x_2-y}  \|\psi(s,x_2-u) \|_pG_{t-s}(u)du.
\end{eqnarray*}
Let $M_{p,T}$ be the constant introduced in (\ref{mp}) and set
\[
V_s:=\sup_{|x-y|= \delta}  \|\psi(s,x) -\psi(s,y) \|_p.
\]
 Then, by Cauchy-Schwarz inequality,
\begin{eqnarray}  \notag
\left\|\int_y^{\infty} \psi(s,z) \big[G_{t-s}(x_1-z)-G_{t-s}(x_2-z)\big]dz\right\|_p  &\le&  V_s+ M_{p,T} |x_1-x_2|^{1/2}  \left(\int_\mathbb{R} G^2_{t-s} (u)du \right)^{1/2}\\
&=&V_s+ M_{p,T}  \sqrt{\delta} [8(t-s)]^{-1/4}.  \label{z1}
\end{eqnarray}
Substituting the estimate (\ref{z1}) into  (\ref{z2}) yields
\begin{eqnarray}  \notag
 \|  I_{2,+}(x_1,x_2 ) \|_p^2   & \le&  2c_p^2 \|f\|^2_{ L^2(\mathbb{R})}  \int_0^t (V_s^2 +8^{-1/2} M_{p,T}^2 \delta (t-s)^{-1/2} )ds\\
 &\leq&  2c_p^2 \|f\|^2_{ L^2(\mathbb{R})}  \left( \int_0^t V_s^2 ds + \sqrt{ \frac T2}  M_{p,T}^2  \delta \right).  \label{z4}
 \end{eqnarray}
An analogous upper bound can be obtained for  $ \|  I_{2,-}(x_1,x_2 ) \|_p^2 $ 
in the same way. Similarly, decompose $I_3$ as
\[
I_3(x_1, x_2)= I_{3,+} (x_1, x_2) + I_{3,-} (x_1,x_2),
\]
where 
\[
 I_{3,+} (x_1, x_2)= \int_0^t  \int_0 ^\infty  \sigma^2_s(y)\int_y ^\infty[G_{t-s}(x_1-z)- G_{t-s}(x_2-z)] \psi(s,z)  dz dyds
 \]
 and
 \[
 I_{3,-} (x_1, x_2)= \int_0^t  \int_{-\infty}^0  \sigma^2_s(y)\int_{-\infty} ^y[G_{t-s}(x_1-z)- G_{t-s}(x_2-z)] \psi(s,z)  dz dyds.
 \]
By  Minkowsky inequality,
\[
 \| I_{3,+} (x_1, x_2)\|_p\le  \int_0^t  \int_0 ^\infty   f^2(y) \left\| \int_y ^\infty[G_{t-s}(x_1-z)- G_{t-s}(x_2-z)] \psi(s,z)  dz  \right \|_p dyds.
 \]
and the estimate (\ref{z1}) leads to
\begin{equation} \label{z5}
 \| I_{3,+} (x_1, x_2)\|_p\le \| f\|^2 _{L^2({\mathbb{R})}}  \left (  \int_0^t    V_s ds +  \frac 43  T^{3/4} M_{p,T} 8^{-1/4} \sqrt{\delta} \right).
  \end{equation}
We can derive an analogous estimate for  $ \| I_{3,-} (x_1, x_2)\|_p$. Finally, from  (\ref{z3}),  (\ref{z4}), (\ref{z5}), and the similar estimates for $I_{2,-}$ and $I_{3,-}$, we deduce
\[
V_t^2 \le C_1 \delta + C_2 \int_0^t  V_s^2 ds
\]
for some constants $C_1$ and $C_2$ depending on  $p$, $T$,  $\|u_0\|_\infty$, $\|u_0\|_{L^1(\mathbb{R})}$ and $\| f\|_{L^2(\mathbb{R})}$.
By Gronwall's lemma, $V_t\le  C  \sqrt{\delta}$, which implies the desired H\"older continuity in the space variable.


  For time regularity, let $0 \leq t_1<t_2 \leq T$ and consider each of the  decomposition
\begin{align*}
\psi(t_2,x)-\psi(t_1,x) &= J_1(t_1,t_2)-\frac{1}{2} J_2(t_1,t_2)+\frac{1}{8} J_3(t_1,t_2),
\end{align*}
where
\begin{align*}
J_1(t_1,t_2) &=\int_{\mathbb{R}} \big( G_{t_2}(x-y)-G_{t_1}(x-y)\big)\psi_0(y) dy,
\end{align*}
\begin{eqnarray*}
J_2(t_1,t_2) &=&
\int_0^{t_2} \int_{S} {\rm sign} (y)G_{t_2-s}(x-z) \psi(s,z) \sigma_s(y) dz W(ds,dy)\\
&&-\int_0^{t_1} \int_{S} {\rm sign} (y)G_{t_1-s}(x-z) \psi(s,z) \sigma_s(y) dz W(ds,dy),
\end{eqnarray*}
and
\begin{eqnarray*}
J_3(t_1,t_2) &=&
\int_0^{t_2} \int_{S} [G_{t_2-s}(x-z)- G_{t-s} (x-z)] \psi(s,z) \sigma^2_s(y)  dz dyds\\
&&-
\int_0^{t_1} \int_{S} [G_{t_1-s}(x-z)- G_{t-s} (x-z)] \psi(s,z) \sigma^2_s(y)  dz dyds.
\end{eqnarray*}
 Apply the semigroup property and the Lipschitz property of $\psi_0$ to get
\begin{align*}
|J_1(t_1,t_2)|  &=   \Big|\int_{\mathbb{R}} G_{t_1}(x-y) \Big( \int_{\mathbb{R}} G_{t_2-t_1}(y-z) \big[\psi_0(z)-\psi_0(y)\big] dz \Big) dy \Big|\\
&\leq C \int_{\mathbb{R}} G_{t_1}(x-y) \Big( \int_{\mathbb{R}} G_{t_2-t_1}(y-z) |z-y| dz \Big) dy\\
&= C (t_2-t_1)^{1/2}.
\end{align*}
For the stochastic integral term, we again decompose $J_2$ as 
$$
J_2(t_1,t_2) = J_{2,+}(t_1,t_2)-J_{2,-}(t_1,t_2),
$$
where
\begin{align*}
 J_{2,+}(t_1,t_2) = &\int_0^{t_2} \int_0^{\infty}\int_y^{\infty} G_{t_2-s}(x-z) \psi(s,z) \sigma_s(y) W(ds,dy)\\&-\int_0^{t_1} \int_0^{\infty}\int_y^{\infty} G_{t_1-s}(x-z) \psi(s,z) \sigma_s(y) W(ds,dy)
\end{align*}
and
\begin{align*}
 J_{2,-}(t_1,t_2) = &\int_0^{t_2} \int_{-\infty}^0\int_{-\infty}^y G_{t_2-s}(x-z) \psi(s,z) \sigma_s(y) W(ds,dy)\\&-\int_0^{t_1} \int_{-\infty}^0 \int_{-\infty}^y G_{t_1-s}(x-z) \psi(s,z) \sigma_s(y) W(ds,dy).
\end{align*}
Splitting $J_{2,+}$ into two pieces, we can write
\begin{eqnarray*}
&& \|J_{2,+}(t_1,t_2)\|_p  \\
&& \quad \leq  \Big\| \int_0^{t_1} \int_0^{\infty} \sigma_s(y) \Big( \int_y^{\infty}  \psi(s,z) [G_{t_2-s}(x-z)-G_{t_1-s}(x-z)]dz \Big) W(ds,dy) \Big\|_p \\
&& \quad \quad + \Big\|\int_{t_1}^{t_2} \int_0^{\infty} \sigma_s(y) \Big( \int_y^{\infty} \psi(s,z) G_{t_2-s}(x-z)dz \Big) W(ds,dy)\Big\|_p\\
&&  \quad =:  A_1(t_1,t_2)  + A_2(t_1,t_2).
\end{eqnarray*}
Applying Burkholder's inequality and Minkowski's inequality, yields
\begin{align*}
A_1(t_1, t_2) & \leq c_p \bigg(\int_0^{t_1} \int_0^{\infty} f(y)^2 \Big\| \int_y^{\infty}  \psi(s,z) [G_{t_2-s}(x-z)-G_{t_1-s}(x-z)]dz \Big\|_p^2 dy ds\bigg)^{1/2}.
\end{align*}
Adding and subtracting $\psi(s,x)$ and  using the spatial regularity of $\psi$, we obtain
\begin{eqnarray}  \notag
&&  \Big\| \int_y^{\infty}  \psi(s,z) [G_{t_2-s}(x-z)-G_{t_1-s}(x-z)]dz \Big\|_p^2 \\ \notag
&& \quad \leq 
  2C \Big(\int_y^{\infty} |G_{t_2-s}(x-z)-G_{t_1-s}(x-z)|~|x-z|^{1/2} dz\Big)^2 \\  \label{k3}
&& \quad \quad +2\| \psi(s,x)\|_p^2 \Big(\int_y^{\infty} \big[G_{t_2-s}(x-z)-G_{t_1-s}(x-z)\big] dz\Big)^2  .
\end{eqnarray}
By Lemma \ref{lem1}, with $\beta =2$, $\theta_1=1$ and $\theta_2 =1/2$, yields
\begin{equation}  \label{k1}
\int_0^{t_1} \left(\int_y^{\infty} |G_{t_2-s}(x-z)-G_{t_1-s}(x-z)|~|x-z|^{1/2} dz\right)^2   ds \le C (t_2-t_2)^{3/2}.
\end{equation}
 Applying  Lemma \ref{lem1} again, with $\beta =2$, $\theta_1=1$ and $\theta_2 =0$, we obtain
 \begin{equation}  \label{k2}
 \int_0^{t_1} \left( \int_y^{\infty} \big[G_{t_2-s}(x-z)-G_{t_1-s}(x-z)\big] dz\right)^2 ds \leq   C(t_2-t_1).  
 \end{equation}
  Substituting (\ref{k1}) and (\ref{k2}) into  (\ref{k3}), we get
  \begin{align*}
A_1(t_1,t_2) & \le C (t_2-t_1) ^{1/2}.
   \end{align*}
We control the  term $A_2(t_1,t_2)$ using a rough estimate as follows
\begin{align*}
 A_2(t_1,t_2) 
&\leq c_p \bigg\| \int_{t_1}^{t_2} \int_0^{\infty} f(y)^2 \Big( \int_y^{\infty} \psi(s,z) G_{t_2-s}(x-z)dz \Big)^2 dy ds \bigg\|_{p/2}^{1/2}\\
&\leq c_p \|f\|_{L^2(\mathbb{R})} \bigg( \int_{t_1}^{t_2} \Big\| \int_{\mathbb{R} }\psi(s,z) G_{t_2-s}(x-z) dz \Big\|_p^2ds \bigg)^{1/2}\\
& \leq C (t_2-t_1)^{1/2}.
\end{align*}
  We can bound $J_{2,-}$ in the same way and  get
$$
\|J_2(t_1,t_2)\|_p \leq C (t_2-t_1)^{1/2}.
$$
Once again, we decompose $J_3$ as $J_3 = J_{3,+}+J_{3,-}$, where
\begin{align*}
J_{3,+}(t_1,t_2) &= \int_0^{t_2}\int_0^{\infty} \int_y^{\infty} G_{t_2-s}(s,z) \psi(s,z) \sigma_s^2(y) dz dy ds \\
&\hspace{10mm} -  \int_0^{t_1} \int_0^{\infty} \int_y^{\infty} G_{t_1-s}(s,z)\psi(s,z) \sigma_s^2(y) dz dy ds
\end{align*}
and
\begin{align*}
J_{3,-}(t_1,t_2) &= \int_0^{t_2}\int_{-\infty}^0 \int_{-\infty}^y G_{t_2-s}(s,z) \psi(s,z) \sigma_s^2(y) dz dy ds \\
&\hspace{10mm} -  \int_0^{t_1} \int_{-\infty}^0 \int_{-\infty}^y  G_{t_1-s}(s,z)\psi(s,z) \sigma_s^2(y) dz dy ds.
\end{align*}
We control $J_{3,+}$ in the same way as $J_{2,+}$ to get
\begin{align*}
\| J_{3,+}(t_1,t_2)\|_p  &\leq \Big\| \int_0^{t_1} \int_0^{\infty} \sigma_s(y)^2 \int_y^{\infty}  \psi(s,z) [G_{t_2-s}(x-z)-G_{t_1-s}(x-z)]dz dy ds \Big\|_p \\
&\hspace{10mm}+ \Big\|\int_{t_1}^{t_2} \int_0^{\infty} \sigma_s(y)^2 \int_y^{\infty} \psi(s,z) G_{t_2-s}(x-z)dz  dy ds \Big\|_p.
\end{align*}
We bound the second term roughly as
\begin{align*}
\Big\|\int_{t_1}^{t_2} \int_0^{\infty} \sigma_s(y)^2 \int_y^{\infty} \psi(s,z) G_{t_2-s}(x-z)dz  dy ds \Big\|_p &\leq C \int_{t_1}^{t_2} \int_{\mathbb{R}}  G_{t_2-s}(x-z) dz  ds\\
&= C (t_2-t_1).
\end{align*}
Then, notice that for any $\epsilon \in (0,1)$ the first term can be bounded as
 \begin{align*}
 \int_0^{t_1} &\int_{\RR} | G_{t_2-s}(x-z)- G_{t_1-s} (x-z)| dz ds \\
 &
\le C \int_0^{t_1} \Big(\int_{\mathbb{R}} |G_{t_2-s}(x)-G_{t_1-s}(x)|^{p_1(1-\epsilon)}dx\Big)^{1/p_1} \\
&\hspace{10mm} \times \Big[\Big(\int_{\RR}G_{t_2-s}(x)^{p_2\epsilon}dx\Big)^{1/p_2}+\Big(\int_{\RR}G_{t_1-s}(x)^{p_2\epsilon}dx\Big)^{1/p_2} \Big]ds\\
&\leq C \int_0^{t_1} \Big(\int_{\mathbb{R}} |G_{t_2-s}(x)-G_{t_1-s}(x)|^{p_1(1-\epsilon)}dx\Big)^{1/p_1}, 
\end{align*}
for any H\"older conjugates $p_1,p_2$. Notice that if $\beta = 1/p_1$, $\theta_1 = p_1(1-\epsilon)$, and $\theta_2 = 0$, then condition (\ref{kk1}) is satisfied when, for example, $\epsilon=1/p_1$ and $p_1>4$. Hence, using Lemma \ref{lem1} with these parameters yields
 $$
  \int_0^{t_1} \int_{\RR} | G_{t_2-s}(x-z)- G_{t_1-s} (x-z)| dz ds \leq C(t_2-t_1)^{1/2+1/p_1}.
 $$
 Control $J_{3,-}$ in an identical way to obtain
$$
\|J_3(t_1,t_2)\|_p \leq C (t_2-t_1)^{1/2}.
$$
Combining the above estimates yields
$$
\|\psi(t_2,x)-\psi(t_1,x)\|_p \leq C (t_2-t_1)^{1/2}.
$$
\end{proof}

Next we use the established H\"older regularity of the process $\psi$ to study its spatial differentiability. 

\begin{proposition}
The process $\psi(t,\cdot)$ is differentiable in $L^p(\Omega)$ for any $p \geq 2$ and satisfies 
\begin{align} \notag
\frac{\partial \psi}{\partial x}(t,x) = \int_{\mathbb{R}} \frac{\partial G_t }{\partial x}(x-y) \psi_0(y) dy -\frac{1}{2}\int_0^t \int_S {\rm sign}(y) \frac{\partial G_{t-s}}{\partial x} (x-z) \psi(s,z)  \sigma_s(y) dz W(ds,dy) \\+ \frac{1}{8}\int_0^t \int_S  \frac{\partial G_{t-s}}{\partial x} (x-z) \psi(s,z) \sigma_s(y)^2 dz dy ds.
\label{kk9}
\end{align} 
\end{proposition}
\begin{proof}
It is clear that the spatial derivative of the first integral in the expression of $\psi$ equals the first integral above by Leibniz's rule.

To take care of the stochastic integral term, by the Burkholder-Davis-Gundy inequality and the symmetry of $S$, it suffices to show the  convergence to zero
in $L^{p/2}(\Omega)$, as $h$ tends to zero, of  the term
$$
 I_h(t,x):=\int_0^t \int_0^{\infty} \Big(\int_y^{\infty}  \Delta_h G_{t-s}(x-z) \psi(s,z) dz \Big)^2 \sigma_s(y)^2 dy ~ ds,
$$
where
\[
 \Delta_h G_{t-s}(x-z) := \frac{G_{t-s}(x+h-z)-G_{t-s}(x-z)}{h}-\frac{\partial G_{t-s}}{\partial x}(x-z).
 \]
By Minkowski's inequality,  we obtain
\[
\|I_h(t,x)\|_{p/2} \le \int_0^t \int_0^{\infty} \left \|\int_y^{\infty}   \Delta_h G_{t-s}(x-z) \psi(s,z) dz \right\|_p^2 f(y)^2 dy ds.\\
\]
We show first the convergence to zero of  
\[
I_{h}(t,x,s):= \int_0^{\infty} \left\|\int_y^{\infty}   \Delta_h G_{t-s}(x-z) \psi(s,z) dz \right\|_p^2 f(y)^2 dy 
\]
 as $h$ tends to zero, for each fixed $s\in [0,t)$. Rough estimates of $I_h(t,x,s)$ lead to
 $$
 I_h(t,x,s) \leq \sup_{t,x} \|\psi(t,x)\|_p^2 \|f\|_{L^2(\mathbb{R})}^2 \left( \int_{\mathbb{R}}| \Delta_h G_{t-s}(z)| dz\right )^2.
 $$
Apply the mean value theorem twice to see that 
$$
 \Delta_h G_{t-s}(z) = \frac{1}{h} \int_0^h \int_0^u \frac{\partial^2 G_{t-s}}{\partial x^2}(z+\eta) d \eta du .
$$
Finally, by applying Fubini's theorem, we obtain
$$
I_{h}(t,x,s) \leq \ C_{s} |h|^2.
$$
 Hence, we have that, for each $s\in [0,t)$,
$
I_{h} (t,x,s) \to 0
$
as $h\to 0$. By the dominated convergence theorem, it now suffices to show that $ I_{h} (t,x,s)$ is bounded by a $ds$-integrable function which is independent of $h$. Again, by the mean  value theorem, we can write
$$
I_h(t,x,s)= \int_0^{\infty}  \left\| \int_y^{\infty} \frac{1}{h} \int_0^h \left( \frac{\partial G_{t-s}}{\partial x}(x+\xi-z)-\frac{\partial G_{t-s}}{\partial x}(x-z)\right) d\xi ~ \psi(s,z) dz
\right \|_p^2  f(y)^2 dy.
$$
We split up this quantity by adding and subtracting appropriate terms as follows  
$$
 I_h(t,x,s)=\int_0^{\infty} \left\| \int_y^{\infty}\big[\phi_1(s,x,z,h) +\phi_2(s,x,z,h)\big] dz \right\|_p^2  f(y)^2 dy,
$$
where
\begin{align*}
\phi_1(s,x,z,h) &:= \frac{1}{h} \int_0^h \frac{\partial G_{t-s}}{\partial x}(x+\xi-z)\left[\psi(s,z)-\psi(s,x+\xi)\right]d\xi\\
 \\&  \quad  - \frac{\partial G_{t-s}}{\partial x}(x-z) \left[\psi(s,z)-\psi(s,x)\right] 
\intertext{and}
\phi_2(s,x,z,h) &:= \frac{1}{h} \int_0^h \left[\frac{\partial G_{t-s}}{\partial x}(x+\xi-z)\psi(s,x+\xi)  - \frac{\partial G_{t-s}}{\partial x}(x-z) \psi(s,x)\right] d\xi.
\end{align*}
Let us first consider the two terms of $\phi_1$, one at a time.  For the first one, we can write, using Minkowski inequality and  the H\"older continuity in $L^p(\Omega)$ of $\psi$
\begin{align*}
 \int_0^{\infty} \Big\|& \int_y^{\infty} \frac{1}{h} \int_0^h \frac{\partial G_{t-s}}{\partial x}(x+\xi-z)\big[\psi(s,z)-\psi(s,x+\xi)\big]d\xi ~ dz \Big\|_p^2 f(y)^2 dy \\
& \leq C  \| f\|^2_{L^2(\mathbb{R})} \frac{1}{h^2} \Big(  \int_0^h \int_{\mathbb{R}} \Big| \frac{\partial G_{t-s}}{\partial x}(x+\xi-z)\Big| |x+\xi-z|^{1/2}  dz~ d\xi \Big)^2\\
&= C \| f\|^2_{L^2(\mathbb{R})} (t-s)^{-1/2},
\end{align*}
which is $ds$-integrable.
Now, to see that the second term  is also bounded by a $ds$-integrable function not depending on $h$, we bound in the same way to get
\begin{align*}
 \int_0^{\infty} \Big\|& \int_y^{\infty}  \frac{\partial G_{t-s}}{\partial x}(x-z)\big[\psi(s,z)-\psi(s,x)\big] dz \Big\|_p^2 f(y)^2 dy \\
 &\leq C \int_0^{\infty} \Big( \int_y^{\infty}  \Big|\frac{\partial G_{t-s}}{\partial x}(x-z)\Big| |z-x|^{1/2}  dz \Big)^2 f(y)^2 dy\\
 &\leq C  \| f\|^2_{L^2(\mathbb{R})}(t-s)^{-1/2}.
\end{align*}
Let us now control the term  $\phi_2$ by first interchanging the $d\xi$ and $dz$ integrals to get 
\begin{align*}
 \int_0^{\infty} &\Big\| \int_y^{\infty} \phi_2(s,x,z,h) dz \Big\|_p^2  f(y)^2 dy  \\
 &= \int_0^{\infty} \Big\|  \frac{1}{h}\int_0^h \Big[G_{t-s}(x+\xi-y)\psi(s,x+\xi) - G_{t-s}(x-y)\psi(s,x)\Big] d\xi  \Big\|_p^2  f(y)^2 dy .
 \intertext{Now, add and subtract $G_{t-s}(x+\xi-y)\psi(s,x)$ to get}
 \int_0^{\infty} &\Big\| \int_y^{\infty} \phi_2(s,x,z,h) dz \Big\|_p^2  f(y)^2 dy\\
 &\leq 2\int_0^{\infty} \Big\| \psi(s,x) \frac{1}{h}\int_0^h \Big[G_{t-s}(x+\xi-y) -G_{t-s}(x-y)\Big] d\xi  \Big\|_p^2  f(y)^2 dy  \\
 &\hspace{10mm} + 2\int_0^{\infty} \Big\| \frac{1}{h}\int_0^h G_{t-s}(x+\xi-y)\Big[\psi(s,x+\xi)-\psi(s,x)\Big] d\xi  \Big\|_p^2  f(y)^2 dy \\
 & =: J_{1,h} +J_{2,h}.
\end{align*}
The second term can easily be bounded as follows
$$
J_{2,h} \le C\int_{\mathbb{R}} f(y)^2 \Big| \frac{1}{h} \int_0^h G_{t-s}(x+\xi-y) |\xi|^{1/2} d\xi \Big|^2 dy.
$$
We now use the assumption  $f \in L^q(\mathbb{R})$ for some $q > 2$ and choose $p_1$ such that $\frac 1{p_1} + \frac 2 q =1$. 
Then, by H\"older's inequality, we can write
$$
J_{2,h} \le C \|f \|_{L^{q}(\mathbb{R})}^2 \Big\| \frac{1}{h} \int_0^h G_{t-s}(x+\xi-\cdot ) |\xi|^{1/2} d\xi \Big\|_{L^{2p_1}(\mathbb{R})}^2.
$$
 Now, by Minkowski's inequality, we have
$$
\Big\| \frac{1}{h} \int_0^h G_{t-s}(x+\xi-\cdot ) |\xi|^{1/2} d\xi \Big\|_{L^{2p_1}(\mathbb{R})}^2 \leq 
 C (t-s)^{-1+1/2p_1}\Big( \frac{1}{h} \int_0^h  |\xi|^{1/2} d\xi \Big)^2,
$$
which is $ds$-integrable and independent of $h$ since we can assume  $|h|\leq 1$ without loss of generality. Finally, to control $J_{1,h}$, we proceed by again choosing the same value of $p_1$:
\begin{align*}
 J_{1,h} 
 &\leq C  \|f\|_{L^{q}(\mathbb{R})}^2 \Big(\frac{1}{h}\int_0^h \big\|G_{t-s}(x+\xi-\cdot) -G_{t-s}(x,\cdot)\big\|_{L^{2p_1}(\mathbb{R})} d\xi\Big)^2  \\
  &\leq C  \|f\|_{L^{q}(\mathbb{R})}^2 (t-s)^{-1+1/(2p_1)},
 \end{align*}
which is $ds$-integrable.

For the third integral in the expression of $\partial_x \psi$, we use an identical argument to obtain pointwise convergence to zero. Furthermore, it is easy to bound the $ds$ integrand by an integrable function which is independent of $h$ since
\begin{eqnarray*}
&& \Big\| \int_0^{\infty} \int_y^{\infty} \Delta_h G_{t-s}(x-z) \psi(s,z) \sigma_s(y)^2 dz dy  \Big\|_p \leq  C \|f\|_{L^2(\mathbb{R})}^2 \int_{\mathbb{R}} |\Delta_h G_{t-s}(x-z) | dz \\
&& \qquad\leq C \int_{\mathbb{R}} \Big(\Big|\frac{\partial G_{t-s}}{\partial x}(x+\xi-z) \Big| + \Big|\frac{\partial G_{t-s}}{\partial x}(x-z) \Big| \Big)dz\\
&&\qquad = C (t-s)^{-1/2},
\end{eqnarray*}
where the second inequality follows from the mean value theorem and triangle  inequality.

\end{proof}
In order to obtain a continuity result for the derivative process given above, we first establish uniform moment bounds.
\begin{proposition} \label{prop3.6}
For all integers $p \geq 2$, we have for any $t\ge 0$,
$$
\sup_{x \in  \mathbb{R}} \Big\|\frac{\partial \psi}{\partial x}(t,x) \Big\|_p  \le 
 K c_p  (t\vee 1)^{1-1/q}  \exp\Big( c_p  \| f\|_{L^2(\RR)}^2t  +\frac{1}{2} \|f\|_{L^2(\mathbb{R})}^4 t^2  \Big),
$$
where $c_p$ is the optimal constant in Burkholder's inequality and $K$ is a constant depending on $p$,
 $q$, $\|f\| _{L^q(\RR)} $,  $\|f\| _{L^2(\RR)}$ , $ \|u_0\|_{L^1(\RR)}$ and  $\|u_0\|_\infty$. 
\end{proposition}

\begin{proof}
From the integral equation (\ref{kk9}) satisfied by $\frac{\partial \psi}{\partial x}(t,x)$, we get the decomposition
\begin{equation} \label{kk11}
\frac{\partial \psi}{\partial x}(t,x) = \mathcal{I}_1(t,x)- \mathcal{I}_2(t,x)+\mathcal{I}_3(t,x),
\end{equation}
where
\[
\mathcal{I}_1(t,x)= \int_{\mathbb{R}} \frac{\partial G_t }{\partial x}(x-y) \psi_0(y) dy
\]
\[
\mathcal{I}_2(t,x)=\frac{1}{2}\int_0^t \int_{S}{\rm sign} (y) \frac{\partial G_{t-s}}{\partial x} (x-z) \psi(s,z)  \sigma_s(y) dz W(ds,dy)
\]
and
\[
\mathcal{I}_3(t,x)= \frac{1}{8}\int_0^t \int_{S} \frac{\partial G_{t-s}}{\partial x} (x-z) \psi(s,z) \sigma_s(y)^2 dz dy ds.
\]
First observe that integrating by parts yields
$$
| \mathcal{ I}_1(t,x)| \leq \int_{\mathbb{R}} G_t(x-y) \bigg|\frac{\partial \psi_0}{\partial y}(y)\bigg| dy \leq \frac{1}{2} \|u_0\|_{\infty} e^{\frac{1}{2}\|u_0\|_{L^1(\mathbb{R})}}  .
$$ 
Now, decompose $\mathcal{I}_2$ as $\mathcal{I}_2(t,x) = \mathcal{I}_{2,+}(t,x)+\mathcal{I}_{2,-}(t,x)$, where
\begin{equation} \label{kl1}
\mathcal{I}_{2,+}(t,x) = \int_0^t \int_0^{\infty}  \int_y^{\infty} \frac{\partial G_{t-s}}{\partial x}(x-z) \psi(s,z)\sigma_s(y) dz W(ds,dy)
\end{equation}
and
\begin{equation} \label{kl2}
\mathcal{I}_{2,-}(t,x) =- \int_0^t \int_{-\infty}^{0}  \int_{-\infty}^y \frac{\partial G_{t-s}}{\partial x}(x-z) \psi(s,z)\sigma_s(y) dz W(ds,dy).
\end{equation}
Using Burkholder's and Minkowski's inequalities, we get 
\begin{align*}
\| \mathcal{I}_{2,+}(t,x) \|_p^2 &\leq c_p \int_0^t \int_0^{\infty} \Big\| \int_y^{\infty} \frac{\partial G_{t-s}}{\partial x}(x-z) \psi(s,z) dz \Big\|_p^2 f(y)^2 dy ds.
\end{align*}
Integrate by parts, use the triangle inequality, and the uniform bounds on moments of $\psi$ to obtain
\begin{align*}
\Big\| \int_y^{\infty} \frac{\partial G_{t-s}}{\partial x}(x-z) \psi(s,z) dz \Big\|_p^2 &\leq 2 M_{p,t}^2 G_{t-s}^2(x-y) \\
&+ 2\Big( \int_{\mathbb{R}} G_{t-s}(x-z) \Big\|\frac{\partial \psi}{\partial z}(s,z) \Big\|_p dz \Big)^2,
\end{align*}
where $M_{p,t} = \sup_{x\in \RR} \| \psi(t,x)\|_p$.
By H\"{o}lder's inequality,  if $\frac 1{ q_1} + \frac 2q =1$, then
\begin{align*}
\int_{\mathbb{R}} f(y)^2 G_{t-s}^2(x-y) dy &\leq \Big( \int_{\mathbb{R}} G_{t-s}(x-y)^{2q_1} dy \Big)^{1/q_1} \|  f\| ^2_{L^q(\mathbb{R})}\\
&= k  \| f\| ^2_{L^q(\mathbb{R})}   (t-s)^{-1+1/(2q_1)},
\end{align*}
where $k$ is a universal constant.
Let $$U_t := \sup_{x\in \mathbb{R}} \Big\|\frac{\partial \psi}{\partial x}(t,x) \Big\|_p^2 .$$
The above estimates yield
\begin{align*}
\| \mathcal{I}_{2,+}(t,x) \|_p^2 &\leq  \frac{ 4q}{q-2}c_p k \| f\| ^2_{L^q(\mathbb{R})} M_{p,t}^2 t^{\frac 12 -\frac 1q} +2c_p\|f\|^2_{L^2(\RR)}\int_0^t U_s ds\\
& = c^{(1)}_{p,t} +  c^{(2)} _{p} \int_0^t U_s ds,
\end{align*}
where $c^{(1)}_{p,t}$ and $c^{(2)}_{p}$  are positive constants depending on $p$, $q$, $t$, $\| f\|_{L^q(\mathbb{R})}$, $\| f\|_{L^2(\mathbb{R})}$ and $\|u_0\|_{L^1(\RR)}$.
 We obtain the same bound on $\|\mathcal{I}_{2,-}(t,x)\|_p$ in an identical way. Similarly, $
\mathcal{I}_3(t,x) = \mathcal{I}_{3,+}(t,x) + \mathcal{I}_{3,-}(t,x)$ where 
\begin{equation} \label{i31}
\mathcal{I}_{3,+}(t,x) = \int_0^t \int_0^{\infty} \int_y^{\infty} \frac{\partial G_{t-s}}{\partial x}(x-z) \psi(s,z)\sigma_s(y)^2 dz dy ds
\end{equation}
and
\begin{equation} \label{i32}
\mathcal{I}_{3,-}(t,x) = \int_0^t \int_{-\infty}^0 \int_{-\infty}^y \frac{\partial G_{t-s}}{\partial x}(x-z) \psi(s,z)\sigma_s(y)^2 dz dy ds.
\end{equation}
Again, integrating by parts, using Minkowski's inequality, and Proposition 3.1, we obtain
\begin{align*}
\| \mathcal{I}_{3,+}(t,x) \|_p &\leq  \int_0^t \int_{0}^{\infty} f(y)^2 \Big( M_{p,t} G_{t-s}(x-y) + \int_y^{\infty} G_{t-s}(x-z) \Big\| \frac{\partial \psi}{\partial z}(s,z) \Big\|_p dz \Big) dy ds\\
&\leq M_{p,t} \|f\|_{L^{q}(\mathbb{R})}^2 \int_0^t \|G_{t-s}\|_{L^{q_1}(\mathbb{R})}ds + \|f\|_{L^2(\mathbb{R})}^2 \int_0^t \sup_{x\in \mathbb{R}} \Big\| \frac{\partial \psi}{\partial x}(s,x) \Big\|_p  ds.
\end{align*}
 Hence, we have
\begin{align*}
\| \mathcal{I}_{3,+}(t,x) \|_p^2 &  \leq    k M_{p,t}^2 \|f\|_{L^q(\mathbb{R})}^4\Big(\frac{q}{q-1}\Big)^2 t^{2-2/q}  + 2t \|f\|_{L^2(\mathbb{R})}^4 \int_0^t U_s ds\\
& = c^{(3)}_{p,t} + c^{(4)} _{p,t} \int_0^t U_s ds,
\end{align*}
for some constants  $c^{(3)}_{p,t}$ and $ c^{(4)} _{p,t}$.
We can bound $\mathcal{I}_{3,-}$ in the same way. Putting each bound from above together and applying Gronwall's inequality, we obtain the desired result. 
\end{proof}

\begin{proposition}  \label{prop3.7}
Suppose that  in addition to condition (A1), the initial condition $u_0$ is H\"older continuous of order $\alpha \in [0,1]$. Then,
for  any $p \geq 2$ and any $T>0$, there exists some constant $C$, depending on $p$, $T$,  $u_0$,  and $ f$, such that for all
 $s,t \in [0,T]$, and   $ x, y \in \mathbb{R}$,  
$$
\Big\| \frac{\partial \psi}{\partial x}(t,x) - \frac{\partial \psi}{\partial y}(s,y)\Big\|_p \leq C   \big(|t-s|^{ \frac {\alpha} 2 \wedge(\frac 14 -\frac 1{2q})}+ |x-y|^{ \alpha \wedge (\frac 12 -\frac 1q) }\big),
$$
where $q$ is the exponent appearing in Assumption (A2).
\end{proposition}
\begin{proof}
We  first study  H\"older continuity in the space variable. 
Fix $t \in [0,T]$, let $x_1,x_2 \in \mathbb{R}$ be given, and set $\delta = |x_1-x_2|$. Without loss of generality we can assume that   $\delta \le 1$.  We consider spatial increments of each term in \eqref{kk11} one at a time. The first term is easily controlled  integrating by parts and using the  fact that   $\psi'_0$ is H\"older continuous of order $\alpha$:
$$
|\mathcal{I}_1(t,x_1)- \mathcal{I}_1(t,x_2)| \leq  \int_{\mathbb{R}} G_t(y) | \psi'_0(x_1-y) - \psi'_0(x_2-y)|dy \leq C \delta^\alpha.
$$
For the second term, we again use the decomposition $\mathcal{I}_{2}(t,x)=\mathcal{I}_{2,+}(t,x)+\mathcal{I}_{2,-}(t,x)$,
where  $\mathcal{I}_{2,+}$ and $\mathcal{I}_{2,-}$  have been introduced in (\ref{kl1}) and (\ref{kl2}), respectively.  Integrating by parts, we obtain
\begin{eqnarray*}  
 && \int_y^{\infty} \Big[ \frac{\partial G_{t-s}}{\partial x}(x_1-z)-\frac{\partial G_{t-s}}{\partial x}(x_2-z)\Big]\psi(s,z) dz \\ 
 && \quad =\big[G_{t-s}(x_1-y) - G_{t-s}(x_2-y) \big] \psi(s,y)\\ 
&&  \qquad  + \int_y^{\infty} \big[G_{t-s}(x_1-z) - G_{t-s}(x_2-z) \big] \frac{\partial \psi}{\partial z}(s,z) dz 
 \\ &&\quad =: \mathcal{I}_{2,+}^A(t-s,x_1, x_2,y) +\mathcal{I}_{2,+}^B(t-s,x_1,x_2,y).
\end{eqnarray*}
Applying Burkholder's inequality,  (\ref{eq3}), Minkowski's inequality, and Proposition 3.1, we get
\begin{eqnarray} \notag
&& \Big\| \int_0^t \int_0^{\infty} \mathcal{I}_{2,+}^A(t-s,x_1,x_2,y) \sigma_s(y) W(ds, dy) \Big\|_p^2 \\ \notag
&& \qquad \leq C \int_0^t \int_{\mathbb{R}}\big| G_{t-s}(x_1-y)-G_{t-s}(x_2-y)\big|^2 f(y)^2 dy ds \\ 
& & \qquad \leq C \|f\|^2_{L^{q}(\mathbb{R})}\  \int_0^t  \left( \int_{\mathbb{R}}  \big|G_{t-s}(x_1-y)-G_{t-s}(x_2-y)\big|^{2q_1}  dy  \right)^{1/q_1}ds,
\label{kl3}
\end{eqnarray}
where $\frac 2q + \frac 1{q_1} =1$. Making the substitutions $y = \delta z$ and $t-s = \delta^2 v$, yields
\begin{eqnarray}   \notag
&&\int_0^t \Big\| \big[G_{t-s}(x_1-\cdot)-G_{t-s}(x_2-\cdot)\big]^2 \Big\|_{L^{q_1}(\mathbb{R})} ds\\  \notag
& & \qquad \leq  C\delta^{1/q_1} \int_0^{\infty} v^{-1} \Big(  \int_{\mathbb{R}} \Big| \exp(-(1+z)^2/4v) - \exp(-z^2/4v)\Big|^{2q_1}dz\Big)^{1/q_1} dv\\      
&& \qquad = C \delta^{1/q_1}.    \label{kl4}
\end{eqnarray}
Therefore, from (\ref{kl3}) and  (\ref{kl4}), we obtain
\begin{equation} \label{kl5}
 \Big\| \int_0^t \int_0^{\infty} \mathcal{I}_{2,+}^A(t-s,x_1,x_2,y) \sigma_s(y) W(ds, dy) \Big\|_p^2
 \le C \|f\|^2_{L^{q}(\mathbb{R})}  \delta^{1/q_1}.
 \end{equation}
To handle $\mathcal{I}_{2,+}^B$, we use the same techniques as in the proof of Proposition 3.4 to first write 
\begin{align*}
\mathcal{I}_{2,+}^B(t-s,x_1,x_2,y) &= \int_{-\infty}^{x_1-y} G_{t-s}(u)\Big[ \frac{\partial \psi}{\partial z}(s,x_1-u)-\frac{\partial \psi}{\partial z}(s,x_2-u)\Big] du \\
&\hspace{10mm}+ \int_{x_2-y}^{x_1-y} \frac{\partial \psi}{\partial z}(s,x_2-u) G_{t-s}(u) du.
\end{align*}
Let 
$$
\widetilde{V}_s := \sup_{|x-y|=\delta} \Big\| \frac{\partial \psi}{\partial x}(s,x)-\frac{\partial \psi}{\partial x}(s,y)\Big\|_p
$$
and 
$$
N_p := \sup_{t,x}\Big\| \frac{\partial \psi}{\partial x}(t,x)\Big\|_p.
$$
Then, we can write 
\begin{align*}
\big\| \mathcal{I}_{2,+}^B(t-s,x_1,x_2,y)\big\|_p &\leq \widetilde{V}_s + N_p \int_{x_2-y}^{x_1-y} G_{t-s}(u) du\\
&\leq \widetilde{V}_s + N_p  \sqrt{\delta} [ 8(t-s)] ^{-1/4}.
\end{align*}
Hence, 
\begin{equation} \label{kl6}
 \Big\| \int_0^t \int_0^{\infty} \mathcal{I}_{2,+}^B(t-s,x_1,x_2,y) \sigma_s(y) W(ds, dy) \Big\|_p^2  
 \le  C  \| f\|^2_{L^2(\mathbb{R})}  \left(\int_0^t \widetilde{V}_s^2 ds +  \delta N_p^2 \right).
\end{equation}
Therefore, from (\ref{kl5}) and (\ref{kl6}), we get
\[
\big\| \mathcal{I}_{2,+}(t,x_1)-\mathcal{I}_{2,+}(t,x_2)\big\|_p^2  
 \leq  C  \|f\|^2_{L^{q}(\mathbb{R})}  \delta^{1/q_1} + C \|f\|_{L^2(\mathbb{R})}^2 \Big(\delta N_p^2  + \int_0^t \widetilde{V}_s^2 ds\Big).
\]
We can get the same bounds on increments of $\mathcal{I}_{2,-}$ in an identical way. Once again, write $\mathcal{I}_3 = \mathcal{I}_{3,+}+\mathcal{I}_{3,-}$, as in (\ref{i31}) and (\ref{i32}). Integrate by parts, and use the same techniques as above to get
\begin{align*}
\big\| \mathcal{I}_{3,+}(t,x_1)- \mathcal{I}_{3,+}(t,x_2)\big\|_p &\leq C \bigg( \int_0^t \int_{\mathbb{R}} \big| G_{t-s}(x_1-y)-G_{t-s}(x_2-y)\big| f(y)^2 dy ds\\
&\qquad  +C\|f\|_{L^2(\mathbb{R})}^2 \sqrt{\delta} N_p +\|f\|_{L^2(\mathbb{R})}^2 \int_0^t  \widetilde{V}_s ds \bigg)\\
&\leq C \Big  ( \| f\|^2 _{ L^q( \mathbb{R})}    +\|f\|_{L^2(\mathbb{R})}^2 \sqrt{\delta} N_p +\|f\|_{L^2(\mathbb{R})}^2 \int_0^t  \widetilde{V}_s ds \Big).
 \end{align*}
 The same bounds for increments of $\mathcal{I}_{3,-}$ are obtained the same way. Put all of these pieces together by taking the smallest power of $\delta$ to get
 $$
 \widetilde{V}_t^2 \leq C \left(  \delta ^{2\alpha} +\delta^{1-  2/q} + \int_0^t \widetilde{V}_s^2 ds \right).
 $$
Thus, Gronwall's inequality implies that $x\mapsto \frac{\partial \psi}{\partial x}(t,x)$ is H\"{o}lder continuous in $L^p(\Omega)$, uniformly in $t$, with order of regularity  $ \alpha \wedge(1/2-  1/q)$.

 To establish regularity in time,  fix $0\le t_1 <t_2 \le T$ and  write
 \begin{eqnarray*}
 | \mathcal{I}_1(t_2,x) - \mathcal{I}_1(t_1,x) | &=&
\left|  \int_{\mathbb{R}} G_{t_1} (x-y) \left(  \int_{\mathbb{R}} G_{ t_2-t_1} (y-z) [ \psi'_0(z) -\psi'_0(y) ] dz \right) dy  \right| \\
 &\le&  C \int_{\mathbb{R}} G_{t_1} (x-y) \left(  \int_{\mathbb{R}} G_{ t_2-t_1} (y-z)  |y-z| ^\alpha dz \right) dy\\
 & =& C (t_2-t_1) ^{\alpha/2}.
 \end{eqnarray*}
 Then, we again split up $\|\mathcal{I}_{2,+}(t_2,x)-\mathcal{I}_{2,+}(t_1,x)\|_p$~ into two terms as 
\begin{eqnarray*}
&& \| \mathcal{I} _{2,+}(t_2,x)- \mathcal{I}_{2,+}(t_1,x)\|_p \\
&& \qquad \leq \Big\| \int_0^{t_1} \int_0^{\infty} \sigma_s(y) \Big( \int_y^{\infty}  \psi(s,z) \frac{\partial }{\partial x}[G_{t_2-s}(x-z)-G_{t_1-s}(x-z)]dz \Big) W(ds,dy) \Big\|_p \\
&&\qquad \qquad + \Big\|\int_{t_1}^{t_2} \int_0^{\infty} \sigma_s(y) \Big( \int_y^{\infty} \psi(s,z) \frac{\partial }{\partial x}G_{t_2-s}(x-z)dz \Big) W(ds,dy)\Big\|_p\\
&& \qquad = : \widetilde{J}_1 + \widetilde{J}_2.
\end{eqnarray*}
Integrate by parts, and apply Burkholder's and Minkowski's inequalities to get
\begin{align*}
\widetilde{J}_2
&\leq c_p \bigg( \int_{t_1}^{t_2} \int_0^{\infty} \Big\| G_{t_2-s}(x-y) \psi(s,y) + \int_y^{\infty} G_{t_2-s}(x-z) \frac{\partial \psi}{\partial z}(s,z) dz \Big\|_p^2 f(y)^2 dy ds \bigg)^{1/2} \\
&\leq C \bigg( \int_{t_1}^{t_2} \int_{\mathbb{R}} G_{t_2-s}(x-y)^2 f(y)^2 dy ds\\
& \qquad \qquad  + \int_{t_1}^{t_2} \int_0^{\infty} f(y)^2 \Big\| \int_y^{\infty} G_{t_2-s}(x-z) \frac{\partial \psi}{\partial z}(s,z) dz \Big\|_p^2 dy ds\bigg)^{1/2}.
\end{align*}
By H\"older's inequality, we have
$$
\int_{t_1}^{t_2} \int_{\mathbb{R}} G_{t_2-s}(x-y)^2 f(y)^2 dy ds \leq \int_{t_1}^{t_2} \|G_{t_2-s}\|_{L^{2q_1}(\mathbb{R})}^2 \|f\|_{L^{q}(\mathbb{R})}^2 ds = C (t_2-t_1)^{1/(2q_1)}
$$
where $ \frac 1{q_1} +\frac 2q =1$. For the other term, we make use of the uniform bounds on moments of the derivative of $\psi$ to get
$$
 \int_{t_1}^{t_2} \int_0^{\infty} f(y)^2 \Big\| \int_y^{\infty} G_{t_2-s}(x-z) \frac{\partial \psi}{\partial z}(s,z) dz \Big\|_p^2dy ds \leq C (t_2-t_1)
$$
for some constant $C$. Hence, 
$$
\widetilde{J}_2 \leq C |t_2-t_1|^{1/(4q_1)}.
$$
For the  term $\widetilde{J}_1$, we first apply Burkholder's inequality and integrate by parts to get
\begin{align*}
\widetilde{J}_1 
&\leq C \bigg( \int_0^{t_1} \int_0^{\infty} f(y)^2 \|\psi(s,y)\|_p^2 \big[ G_{t_2-s}(x-y)-G_{t_1-s}(x-y)\big]^2 dy ds \\
& \qquad + \int_0^{t_1}\int_0^{\infty}\bigg\| \int_y^{\infty} \big[ G_{t_2-s}(x-z)-G_{t_1-s}(x-z)\big] \frac{\partial \psi}{\partial z}(s,z) dz \bigg\|_p^2 f(y)^2 dy ds \bigg)^{1/2}\\
&  =:  C(\widetilde{J}_{1,1} + \widetilde{J}_{1,2})^{\frac 12}.
\end{align*}
Using  the uniform bounds on $\psi$, choosing $q_1$ such that $\frac 1{q_1} + \frac 2q =1$, 
and applying Lemma \ref{lem1} with $\beta=1/q_1$, $\theta_1 =2q_1$ and $\theta_2=0$, we can write 
\beas
\widetilde{J}_{1,1} & \leq &  C \| f\|^2_{L^q(\RR)}  \int_0^{t_1}  \left( \int_{\RR} | G_{t_2-s}(y)-G_{t_1-s}(y) |^{2q_1} dy \right) ^{\frac 1{q_1}} ds \\
& \le & C   \| f\| ^2_{ L^q(\RR)} (t_2-t_1)^{1/(2q_1)}.
\eeas
For the term $\widetilde{J}_{1,2}$, we same techniques as in the proof of the  H\"older regularity in  time of $\psi$ by first adding and subtracting $\displaystyle \frac{\partial \psi}{\partial x}(s,x)$ and applying the spatial regularity of the derivative of $\psi$ to get
\beas
\widetilde{J}_{1,2} 
&\leq & 2 \|f\|_{L^2(\mathbb{R})}^2 \int_0^{t_1} \Big( \int_{\mathbb{R}}|G_{t_2-s}(x-z)-G_{t_1-s}(x-z)| |x-z|^{\alpha \wedge (\frac 12-\frac 1q)} dz \Big)^2 ds \\
&& + 2\int_0^{t_1}\int_0^{\infty} \Big\|\frac{\partial \psi}{\partial x}(s,x)\Big\|_p^2 \Big(\int_y^{\infty} \big[G_{t_2-s}(x-z)-G_{t_1-s}(x-z)\big]dz\Big)^2f(y)^2 dy ds.
\eeas
We apply Lemma \ref{lem1} with $\beta = 2, \theta_1 = 1, $ and $\theta_2 = \alpha\wedge (\frac{1}{2}-\frac{1}{q})$  to get
\[
  \int_0^{t_1} \left( \int_{\mathbb{R}}|G_{t_2-s}(x-z)-G_{t_1-s}(x-z)| |x-z|^{\alpha\wedge(\frac 12-\frac 1q)} dz \right)^2 ds \\
  \le  C  (t_2-t_1) ^{1+ \alpha \wedge (\frac 12-\frac 1q)}.
\]
 Another application of  Lemma \ref{lem1} with $\beta=2$, $\theta_1=1$ and $\theta_2=0$, yields
 \[
  \int_0^{t_1} \left( \int_{y} ^\infty [G_{t_2-s}(x-z)-G_{t_1-s}(x-z)] dz \right)^2 ds \\
  \le  C  (t_2-t_1).
\]
  Hence, 
\begin{align*}
\widetilde{J}_1\leq  C (t_2-t_1)^{1/(4q_1)}.
\end{align*}
 Put these together to get
$$
\| \mathcal{I}_{2,+}(t_2,x)-\mathcal{I}_{2,+}(t_1,x)\|_p \leq C(t_2-t_1)^{1/4-1/(2q)}.
$$
We can obtain the same upper bound for $I_{2,-}$ and hence
$$
\|\mathcal{I}_2(t_2,x)-\mathcal{I}_2(t_1,x)\|_p\leq C(t_2-t_1)^{ 1/4 -1/(2q)}.
$$
For the third term, we apply the same techniques we used for $\mathcal{I}_2$ to get
$$
\| \mathcal{I}_3(t_2,x)-\mathcal{I}_3(t_1,x)\|_p \leq C (t_2-t_1)^{1/4}.
$$
Hence, we have the desired result.
\end{proof}

\begin{remark}  \label{rem2}
If we do not assume the H\"older continuity of $u_0$, then $\psi'_0$ is  only continuous.  Then,  avoiding the integration by parts in the proof of
the H\"older continuity of the first term, we have  a result of the form
$$
\Big\| \frac{\partial \psi}{\partial x}(t,x) - \frac{\partial \psi}{\partial y}(s,y)\Big\|_p \leq C  (t\wedge s) ^{-1/2}  \big(|t-s|^{\frac 14 -\frac 1{2q}}+ |x-y|^{ \frac 12 -\frac 1q }\big),
$$
where the factor $t ^{-1/2}  $, assuming $t\le s$,  comes from the integral  $ \int_{\RR} \left|\frac {\partial G_t}{\partial t} (y) \right| dy$.
That is, the H\"older continuity blows up at $t=0$. However, $\frac{\partial \psi}{\partial x}(t,x)$ is continuous in $L^p(\Omega)$ on $\RR_+\times \RR$ for all $p\ge 2$, because
 $\psi'_0$ is continuous.
\end{remark}


\section{Hopf-Cole Transformation}
In this section,  we construct a solution to  Burgers' equation (\ref{eq1}) using the Hopf-Cole transformation and the results of the previous section.
Notice first that the process
\beas
v(t,x) &:=& -2 \frac{\partial }{\partial x} \log \psi(t,x) =- \frac 2{\psi(t,x)}  \frac {\partial \psi}{\partial x}(t,x)
\eeas 
is well defined and has uniformly bounded moments of order $p$ for all $p\ge 2$, due to  Proposition   \ref{prop3.6}  and Remark \ref{rem1}.
We now establish the main result of the paper which asserts that the process $v(t,x)$ is the solution to the Burgers' equation (\ref{eq1}). Again, uniqueness follows for free from \cite{burger2}. 

The main idea of the proof is to introduce the regularized process
$$
\psi_{\epsilon}(t,x) := \int_{\mathbb{R}} G_{\epsilon}(x-y) \psi(t,y) dy,
$$
for $\ep \in (0,1]$ and to find the equation satisfied by  $u_{\epsilon}(t,x) := -2\frac{\partial }{\partial x} \log \psi_{\epsilon}(t,x)$.
Based on  previous results, it is easy to see that $\psi_{\epsilon}$ satisfies the following property.

\begin{lemma} \label{lem4.1}
For  any $p\geq 2$ and $T>0$,  we have 
\begin{equation} \label{kk3}
\sup_{x \in \RR,\epsilon \in (0,1], t \in [0,T] } \left( \|\psi_{\epsilon}(t,x) \|_p +\| \psi_{\epsilon}(t,x)^{-1} \|_p 
+ \left\| \frac{\partial \psi_\ep}{\partial x}(t,x) \right\| _p\right) < \infty.
\end{equation}
For any $p\ge 2$, $x\in \RR$, and $t\in (0,T]$, we have
\begin{equation} \label{kk4}
  \| \psi(t,x) - \psi_{\epsilon}(t,x)\|_p  \le C \ep^{1/4}
 \end{equation}
 and
 \begin{equation} \label{kk5}
 \left\| \frac{\partial \psi}{\partial x}(t,x) - \frac{\partial \psi_{\epsilon}}{\partial x}(t,x)\right  \|_p \le C t^{-1/2}  \ep^{1/4 -1/(2q)}.
\end{equation}
\end{lemma}
\begin{proof}
Inequality (\ref{kk3}) follows form 
Jensen's inequality,  Propositions \ref{prop1}  and \ref{prop3.6}, and Remark \ref{rem1}.
Inequalities (\ref{kk4}) and (\ref{kk5}) are consequences of  Proposition \ref{prop3.4} and Remark \ref{rem2}.
\end{proof}

\begin{theorem}
The process $v(t,x) = -2 \frac{\partial }{\partial x} \log \psi(t,x)$ is a solution to $(1)$.
\end{theorem}
\begin{proof}

From Proposition  \ref{prop3.2}, we have that $\psi_{\epsilon}$ satisfies 
\begin{align*}
\psi_{\epsilon}(t,x) &= \int_{\mathbb{R}} G_{t+\epsilon}(x-y) \psi_0(y) dy -\frac{1}{2} \int_0^t \int_{S}{\rm sign} (y) G_{t+\epsilon-s}(x-z) \psi(s,z) \sigma_s(y)  dz W(ds,dy) 
\\  &\hspace{5mm}+\frac{1}{8} \int_0^t \int_{S} G_{t+\epsilon-s}(x-z) \psi(s,z)   \sigma_s(y)^2 dz dy ds.
\end{align*}
Next, apply the semigroup property of the heat kernel to get
\begin{align*}
\psi_{\epsilon}(t,x)&= \int_{\mathbb{R}} G_t(x-z) \Big(\int_{\mathbb{R}} G_{\epsilon}(z-y)\psi_0(y) dy\Big) dz\\ 
&\hspace{5mm}-\frac{1}{2} \int_0^t \int_{S} \int_{\mathbb{R}} {\rm sign} (v)G_{t-s}(x-y)G_{\epsilon}(y-z) \psi(s,z) \sigma_s(v) dy dz  W(ds, dv) \\
&\hspace{9mm} + \frac{1}{8} \int_0^t \int_{S} \int_{\mathbb{R}} G_{t-s}(x-y)G_{\epsilon}(y-z) \psi(s,z) \sigma_s(v)^2 dy dz dv ds.
\end{align*}
Note that this is the mild formulation of the following stochastic heat equation
\begin{align*}
\psi_{\epsilon}(t,x) &= \int_{\mathbb{R}} G_{\epsilon}(x-y) \psi_0(y)dy + \int_0^t \frac{\partial^2 \psi_{\epsilon}}{\partial x^2}(s,x) ds \\
& \hspace{5mm} - \frac{1}{2} \int_0^t \int_{S}{\rm sign} (y) G_{\epsilon}(x-z) \psi(s,z) \sigma_s(y) dz W(ds,dy)\\ 
&\hspace{10mm}+ \frac{1}{8} \int_0^t \int_{S} G_{\epsilon}(x-z) \psi(s,z) \sigma_s(y)^2 dz dy ds.
\end{align*}
The process $t \to \psi_\ep(t,x)$ is a semimartingale and applying It\^{o}'s formula to $\log \psi_{\epsilon}(t,x)$  yields
\begin{align*}
\log \psi_{\epsilon}(t,x) &= \log \Big(\int_{\mathbb{R}} G_{\epsilon}(x-y)\psi_0(y) dy\Big) + \int_0^t \frac{1}{\psi_{\epsilon}(s,x)} \frac{\partial^2 \psi_{\epsilon}}{\partial x^2}(s,x) ds \\
& \hspace{5mm} -\frac{1}{2} \int_0^t \int_{S}{\rm sign} (y) \frac{1}{\psi_{\epsilon}(s,x)} G_{\epsilon}(x-z) \psi(s,z) \sigma_s(y) dzW(ds,dy)\\
&  +\frac{1}{8} \int_0^t \int_{S} \frac{1}{\psi_{\epsilon}(s,x)} G_{\epsilon}(x-z) \psi(s,z)  \sigma_s(y)^2 dz dy ds\\
& \hspace{-5mm} -\frac{1}{8} \int_0^t \int_{\mathbb{R}} \frac{1}{\psi_{\epsilon}(s,x)^2}\Psi_{\epsilon}(s,x,y)^2 \sigma_s(y)^2 dy ds,
\end{align*}
where 
$$\Psi_{\epsilon}(s,x,y):=\mathbf{1}_{(y \geq 0)} \int_y^{\infty} G_{\epsilon}(x-z) \psi(s,z) dz +\mathbf{1}_{(y < 0)} \int_{-\infty}^y G_{\epsilon}(x-z) \psi(s,z) dz.$$ 
Now, noting that basic calculus gives $\frac{1}{f} \frac{\partial^2 f}{\partial x^2} = \frac{\partial^2}{\partial x^2} (\log f) + (\frac{\partial }{\partial x} \log f )^2$,  we have
\begin{align*}
\frac{\partial }{\partial x} \log \psi_{\epsilon}(t,x) &= \frac{\partial }{\partial x} \log \Big( \int_{\mathbb{R}} G_{\epsilon}(x-y) \psi_0(y) dy \Big)\\
&+ \int_0^t \frac{\partial^2 }{\partial x^2}\Big(\frac{\partial }{\partial x}\log \psi_{\epsilon}(s,x) \Big) ds + \int_0^t \frac{\partial }{\partial x}\bigg( \Big( \frac{\partial }{\partial x}\log \psi_{\epsilon}(s,x) \Big)^2 \bigg) ds \\
&\hspace{5mm}- \frac{1}{2} \int_0^t \int_{S} {\rm sign} (y)\frac{\partial }{\partial x}\Big( \frac{1}{\psi_{\epsilon}(s,x)} G_{\epsilon}(x-z)\Big) \psi(s,z)   \sigma_s(y) dz W(ds,dy) \\
&\hspace{5mm}+ \frac{1}{8}\int_0^t \int_{S} \frac{\partial }{\partial x}\Big( \frac{1}{\psi_{\epsilon}(s,x)} G_{\epsilon}(x-z)\Big) \psi(s,z)  \sigma_s(y)^2 dz dy ds\\
&\hspace{5mm}- \frac{1}{8}\int_0^t \int_{\mathbb{R}} \frac{\partial }{\partial x}\Big( \frac{1}{\psi_{\epsilon}(s,x)} \Psi_{\epsilon}(s,x,y) \Big)^2 \sigma_s(y)^2 dy ds.
\end{align*}
So, the process $u_{\epsilon}(t,x) := -2\frac{\partial }{\partial x} \log \psi_{\epsilon}(t,x)$ satisfies the following integral equation
\begin{align*}
u_{\epsilon}(t,x) &= \int_{\mathbb{R}} G_t(x-y) u_{\epsilon}(0,y) dy - \frac{1}{2} \int_0^t \int_{\mathbb{R}} G_{t-s}(x-y) \frac{\partial }{\partial y} u_{\epsilon}(s,y)^2 dy ds \\
&+ \int_0^t \int_{S}\int_{\mathbb{R}} {\rm sign} (y)G_{t-s}(x-v) \frac{\partial }{\partial v} \Big( \frac{1}{\psi_{\epsilon}(s,v)}  G_{\epsilon}(v-z)\Big) \psi(s,z)  \sigma_s(y) dv dz W(ds,dy) \\
&-\frac{1}{4} \int_0^t  \int_S \int_{\mathbb{R}} G_{t-s}(x-v)\frac{\partial }{\partial v} \Big( \frac{1}{\psi_{\epsilon}(s,v)}  G_{\epsilon}(v-z) \Big) \psi(s,z) \sigma_s(y)^2 dv dy dz ds\\
&+\frac{1}{4} \int_0^t\int_{\mathbb{R}}\int_{\mathbb{R}} G_{t-s}(x-v)  \frac{\partial }{\partial v} \bigg( \frac{1}{\psi_{\epsilon}(s,v)}\Psi_{\epsilon}(s,v,y) \bigg)^2 \sigma_s(y)^2 dv  dy ds .
\end{align*}
Finally, integration by parts yields 
\beas
u_{\epsilon}(t,x) &=&  \int_{\mathbb{R}} G_t(x-y) u_{\epsilon}(0,y) dy + \frac{1}{2} \int_0^t \int_{\mathbb{R}} \frac{\partial  }{\partial y}G_{t-s}(x-y) u_{\epsilon}(s,y)^2 dy ds \\ 
&&- \int_0^t \int_{S}\int_{\mathbb{R}} {\rm sign} (y)\frac{\partial  }{\partial v}G_{t-s}(x-v) \frac{1}{\psi_{\epsilon}(s,v) }  G_{\epsilon}(v-z) \psi(s,z)  \sigma_s(y) dv dz W(ds,dy) \\
&&+\frac{1}{4} \int_0^t \int_{S} \int_{\mathbb{R}}\frac{\partial }{\partial v}G_{t-s}(x-v) \frac{1}{\psi_{\epsilon}(s,v)}  G_{\epsilon}(v-z) \psi(s,z)  \sigma_s(y)^2   dv dz dy ds \\
&&- \frac{1}{4} \int_0^t \int_{\mathbb{R}} \int_{\mathbb{R}}\frac{\partial }{\partial v}G_{t-s}(x-v) \frac{1}{\psi_{\epsilon}^2(s,v)} \Psi_{\epsilon}(x,v,y)^2 \sigma_s(y)^2 dv dy ds\\
&=& \sum_{i=1}^5 A_{i,\ep}.
\eeas
We will study the convergence of each term in the above expression. This will be done in several steps:

\medskip
\noindent
{\it Step 1. } For the term $A_{1,\ep}$, taking into account that
\[
u_\ep(0,x)= -2   \frac {(\psi'_0 *G_\ep) (x)} {(\psi_0 * G_\ep)(x)}
\]
and $\psi'_0$ is continuous and bounded, it is easy to show that
\[
A_{1,\ep} \rightarrow  \int_{\mathbb{R}} G_t(x-y) u_0(y) dy,
\]
as $\ep$ tends to zero.

\medskip
\noindent
{\it Step 2.} From Lemma \ref{lem4.1} it follows that
\[
\| u_{\epsilon}(t,x) - v(t,x)\|_p \le C t^{-1/2} \ep^{\frac 14-\frac 1{2q}}.
\]
 With this, it is easy to see that 
 $$
A_{2,\ep} \rightarrow \frac{1}{2} \int_0^t \int_{\mathbb{R}} \frac{\partial }{\partial y} G_{t-s}(x-y) v(s,y)^2 dy ds,
 $$
 as $\ep \rightarrow 0$, in $L^p(\Omega)$ for all $p\ge 2$.
 
 \medskip
\noindent
{\it Step 3.} 
We now show the convergence of the stochastic integral term $A_{3,\ep}$. Integrating by parts, first with respect to $v$, then with respect to $z$, we get
for $y>0$,
\begin{align*}
 \int_{\mathbb{R}}\frac{\partial }{\partial v} G_{t-s}(x-v) & \frac{1}{\psi_{\epsilon}(s,v)} \Big( \int_y^{\infty} G_{\epsilon}(v-z) \psi(s,z) dz \Big) dv \\
 = \int_{\mathbb{R}} &G_{t-s}(x-v) \frac{ 1}{\psi_{\epsilon}(s,v)^2} \frac{\partial \psi_{\epsilon}}{\partial v}(s,v) \Big(\int_y^{\infty} G_{\epsilon}(v-z) \psi(s-z) dz \Big) dv \\ &-\int_{\mathbb{R}} G_{t-s}(x-v) \frac{ 1}{\psi_{\epsilon}(s,v)}  \Big(\int_y^{\infty} \frac{\partial G_{\epsilon}}{\partial v} (v-z) \psi(s,z) dz \Big) dv \\
  = \int_{\mathbb{R}} &G_{t-s}(x-v) \frac{ 1}{\psi_{\epsilon}(s,v)^2} \frac{\partial \psi_{\epsilon}}{\partial v}(s,v) \Big(\int_y^{\infty} G_{\epsilon}(v-z) \psi(s,z) dz \Big) dv
 \\ &-\int_{\mathbb{R}} G_{t-s}(x-v) \frac{ 1}{\psi_{\epsilon}(s,v)}  \Big(\int_y^{\infty} G_{\epsilon}(v-z)\frac{\partial \psi}{\partial z} (s,z) dz \Big) dv\\
 & \hspace{5mm} - \psi(s,y) \int_{\mathbb{R}} G_{t-s}(x-v) \frac{1}{\psi_{\epsilon}(s,v)} G_{\epsilon}(v-y) dv\\
 & =: G_{1,+,\ep}(s,y)- G_{2,+, \ep} (s,y) - G_{3,\ep}(y,s).
 \end{align*}
 In a similar way, for $y<0$, we obtain
 \begin{align*}
 \int_{\mathbb{R}}\frac{\partial }{\partial v} G_{t-s}(x-v) & \frac{1}{\psi_{\epsilon}(s,v)} \Big( \int_{-\infty}^y G_{\epsilon}(v-z) \psi(s,z) dz \Big) dv \\
 & =G_{1,-,\ep}(s,y)- G_{2,-, \ep} (s,y) + G_{3,\ep}(y,s),
 \end{align*}
 where the terms $G_{1,-,\ep}(s,y)$ and $G_{2,-, \ep} (s,y) $ are analogous to  $G_{1,+,\ep}(s,y)$ and $G_{2,+, \ep} (s,y) $, respectively, by just replacing the integral $\int _y ^\infty$ by $\int_{-\infty} ^y$.

 We claim that  the following convergences hold in $L^p(\Omega)$, for any $p\ge 2$,   as $\ep \rightarrow 0$:
\begin{equation} \label{ecua1}
\int_0^t \int_{\mathbb{R}}  G_{3,\ep} (s,y)\sigma_s(y) W(ds, dy) \rightarrow \int_0^t \int_{\mathbb{R}} G_{t-s}(x-y) \sigma_s(y) W(ds, dy),
\end{equation}
\begin{equation} \label{ecua2}
\int_0^t \int_{\mathbb{R}}  [G_{1,+,\ep} (s,y)- G_{2,+,\ep} (s,y)]\sigma_s(y) W(ds, dy) \rightarrow 0.
\end{equation}
and
\begin{equation} \label{ecua3}
\int_0^t \int_{\mathbb{R}}  [G_{1,-,\ep} (s,y)- G_{2,-,\ep} (s,y)]\sigma_s(y) W(ds, dy) \rightarrow 0.
\end{equation}

\noindent
{\it Proof of (\ref{ecua1}):}
Applying Burkholder's inequality and Minkowski's inequality, we can write
\begin{align*}
\Big\| \int_0^t \int_{\mathbb{R}^2}& G_{\epsilon}(v-y) \Big(\frac{\psi(s,y)}{\psi_{\epsilon}(s,v)} G_{t-s}(x-v) - G_{t-s}(x-y) \Big)\sigma_s(y) dv W(ds,dy) \Big\|_p^2\\
&\leq C\int_0^t \int_{\mathbb{R}} \Big\| \int_{\mathbb{R}} G_{\epsilon}(v-y) \frac{\psi(s,y)}{\psi_{\epsilon}(s,v)} G_{t-s}(x-v) dv - G_{t-s}(x-y) \Big\|_p^2 f(y)^2 dy ds\\
& \leq C( B_{1,\ep} +B_{2,\ep}),
\end{align*}
where
\[
B_{1,\ep}= \int_0^t \int_{\mathbb{R}} \left(  \int_{\mathbb{R}} G_{\epsilon}(v-y)  G_{t-s}(x-v)
 \left\| \frac{\psi(s,y)}{\psi_{\epsilon}(s,v)} -1\right\|_p  dv \right)^2 f(y)^2 dy ds
 \]
and
\[
B_{2,\ep}=  \int_0^t \int_{\mathbb{R}}  ( G_{t-s+\epsilon}(x-y)-  G_{t-s}(x-y))^2 
     f(y)^2 dy ds.
     \]
Using the definition of $\psi_{\epsilon}$ and \ref{prop3.4}, it is not difficult to see that $\psi_{\epsilon}$ is H\"older continuous of order 1/2 in the spatial variable.  With this and Lemma \ref{lem4.1}, we have
$$
\left\| \frac{\psi(s,y)}{\psi_{\epsilon}(s,v)} -1\right\|_p \leq C (\epsilon^{1/4} + |y-v|^{1/2}).
$$ 
Therefore, 
\begin{align*}
B_{1,\epsilon} &\leq C \epsilon^{1/2} \int_0^t \int_{\mathbb{R}} G_{t+\epsilon-s}^2(x-y)f(y)^2 dy ds \\
&\hspace{10mm}+C \int_0^t \int_{\mathbb{R}} \Big( \int_{\mathbb{R}} G_{\epsilon}(v-y)G_{t-s}(x-v)|v-y|^{1/2} dv \Big)^2 f(y)^2 dy ds.
\end{align*}
Clearly,
$$
\int_0^t \int_{\mathbb{R}} G_{t+\epsilon-s}^2(x-y) f(y)^2 dy ds \leq C,
$$
by H\"older's inequality and assumption (A.2). Next, make the change of variables $v-y=z$ and choose $q_1>1$ such that $\frac 1{q_1} + \frac 2q=1$, to get
\begin{align*}
\int_0^t \int_{\mathbb{R}} &\Big( \int_{\mathbb{R}} G_{\epsilon}(v-y)G_{t-s}(x-v)|v-y|^{1/2} dv \Big)^2 f(y)^2 dy ds\\
 &\leq \|f\|_{L^q(\mathbb{R})}^2 \int_0^t \Big( \int_{\mathbb{R}} G_{\epsilon}(z) |z|^{1/2} dz\Big)^2\|G_{t-s}\|_{L^{2q_1}(\mathbb{R})}^2 ds \\
 &\leq C \epsilon^{1/2}.
\end{align*}
Hence, $B_{1,\epsilon} \to 0$ as $\epsilon \to 0$. On the other hand, again using Lemma \ref{lem1}, yields
\[
B_{2,\ep} \le C  \| f\|^2_{L^q(\RR)} \ep^{1/2-1/q}.
\]

\noindent
{\it Proof of (\ref{ecua2}):}
Adding and subtracting $\psi(s,v)$ and $\frac{\partial \psi}{\partial v}(s,v)$ in the $dz$ integrals of the first and second terms, respectively, we get 
the decomposition
\[
\int_0^t \int_{\mathbb{R}}  [G_{1,+,\ep} (s,y)- G_{2,+,\ep} (s,y)]\sigma_s(y) W(ds, dy) 
= J_{1,\ep} + J_{2,\ep} + J_{3,\ep},
\]
where
\[
J_{1,\ep}= \int_0^t \int_{0}^{\infty}  \int_{\mathbb{R}} \frac{G_{t-s}(x-v)}  {\psi^2_{\epsilon}(s,v)}
 \frac{\partial \psi_{\epsilon}}{\partial v}(s,v) 
\left( \int_y^{\infty} G_{\epsilon}(v-z)\big[ \psi(s,z)-\psi(s,v)\big] dz  \right) dv \sigma_s(y) W(ds,dy),
\]
\[
 J_{2,\ep}=   \int_0^t \int_{0}^{\infty}  \int_{\mathbb{R}} \frac{G_{t-s}(x-v)}{\psi_{\epsilon}(s,v)}
\left(  \int_y^{\infty} G_{\epsilon}(v-z) \Big[\frac{\partial \psi}{\partial z}(s,z)- \frac{\partial \psi}{\partial v}(s,v)\Big]dz \right)dv \sigma_s(y) W(ds,dy) 
\]
and 
\begin{align*}
 J_{3,\ep}= \int_0^t \int_{0}^{\infty}  \int_{\mathbb{R}}   \frac {G_{t-s}(x-v) } {\psi_{\epsilon}^2(s,v)} \Big[\psi(s,v)\frac{\partial \psi_{\epsilon}}{\partial v}(s,v)
 -\psi_{\epsilon}(s,v)\frac{\partial \psi}{\partial v}(s,v)\Big]\\ 
 \times \Big( \int_y^{\infty} G_{\epsilon}(v-z) dz\Big)dv~ \sigma_s(y) W(ds,dy). 
\end{align*}
Applying Burkholder and  Minkowski inequalities yields, for any $p\ge 2$,
\beas
\|J_{1,\ep}\|_p^2 & \le& c_p 
 \int_0^t \int_{0}^{\infty}   \bigg(  \int_{\mathbb{R}} G_{t-s}(x-v) \left\|   \frac 1{\psi^2_{\epsilon}(s,v)}
 \frac{\partial \psi_{\epsilon}}{\partial v}(s,v)  \right\|_{2p} \\
 && \times 
 \int_{\RR}  G_{\epsilon}(v-z) \left\|  \psi(s,z)-\psi(s,v) \right\|_{2p} dzdv  \bigg) ^2  f^2(y)dyds.
\eeas
By Lemma \ref{lem4.1} and Proposition \ref{prop3.4}, we obtain
\beas
\|J_{1,\ep}\|_p^2  &\le& c_p   \| f\|^2_{L^2(\RR)} \int_0^t  \left(  \int_{\mathbb{R}} G_{t-s}(x-v) 
 \int_{\RR}  G_{\epsilon}(v-z)  |z-v|^{1/2}  dzdv  \right) ^2  ds\\
&\le & C\ep^{1/2}.
\eeas

For the term $J_{2,\ep}$ we can write, using Burkholder and Minkowski inequalities and applying Lemma \ref{lem4.1}
\beas
\|J_{2,\ep}\|_p^2 & \le& c_p 
 \int_0^t \int_{0}^{\infty}   \bigg(  \int_{\mathbb{R}} G_{t-s}(x-v) \left\|   \psi^{-1}_{\epsilon}(s,v)
  \right\|_{2p}  \\
  &&\times 
 \int_{\RR}  G_{\epsilon}(v-z) \left\|   \frac{\partial \psi}{\partial z}(s,z)- \frac{\partial \psi}{\partial v}(s,v) \right\|_{2p} dzdv  \bigg) ^2  f^2(y)dyds\\
 &\le& c_p   \| f\|^2_{L^2(\RR)}
 \int_0^t   \bigg(  \int_{\mathbb{R}^2} G_{t-s}(x-v) 
   G_{\epsilon}(v-z) \left\|   \frac{\partial \psi}{\partial z}(s,z)- \frac{\partial \psi}{\partial v}(s,v) \right\|_{2p} dzdv  \bigg) ^2  ds.
\eeas
By the continuity of $(s,z) \rightarrow  \frac{\partial \psi}{\partial z}(s,z)$ in $L^p$, for any $p\ge 2$, in $[0,t] \times \RR$, established in
Remark \ref{rem2}, it follows that the integrand of the above integral on $[0,t]$ converges to zero for any $s\in [0,t]$. On the other hand, the integrand is bounded by an integrable function, which does not depend on $\ep$. Therefore, by the dominated convergence theorem, we conclude that 
$\|J_{2,\ep}\|_p^2$ converges to zero as $\ep$ tends to zero.

Finally for $J_{3,\ep}$,  using Burkholder and Minkowski inequalities and applying Lemma \ref{lem4.1}, we have
\beas
\|J_{3,\ep}\|_p^2 & \le& c_p   \| f\|^2_{L^2(\RR)}
 \int_{0}^{t}   \bigg(  \int_{\mathbb{R}} G_{t-s}(x-v) \left\|    \psi^{-2}_{\epsilon}(s,v)
  \right\|_{2p}  \\
  &&\times 
  \left\|   \psi(s,v)  \frac{\partial \psi_\ep}{\partial v}(s,v)-  \psi_\ep(s,v)\frac{\partial \psi}{\partial v}(s,v) \right\|_{2p} dv  \bigg) ^2  ds.
  \eeas
For $(s,v) \in (0,t) \times \RR$, the term  $  \left\|   \psi(s,v)  \frac{\partial \psi_\ep}{\partial v}(s,v)-  \psi_\ep(s,v)\frac{\partial \psi}{\partial v}(s,v) \right\|_{2p}$
converges to zero as $\ep$ tends to zero, due to the estimates (\ref{kk4}) and (\ref{kk5}).  Therefore, by the dominated convergence theorem
we conclude that $ \|J_{3,\ep}\|_p^2$ tends to zero as $\ep$ tends to zero.
The proof of (\ref{ecua3}) is similar and omitted.

\medskip
\noindent
{\it Step 4.}
Finally, we show that  $ A_{4,\ep} + A_{5,\ep}$ converges to zero  in $ L^p(\Omega)$  for all $p\ge 2$, as $\ep$ tends to zero.
 Once again, we show convergence of the terms when $ z \geq y \geq 0$. When $z \leq y \leq 0$, the proof follows in the same way.  
 The contribution of $\{y>0\}$ can be expressed as follows
  \begin{align*}
&& H_\ep:= \int_0^t \int_0^{\infty} \int_{\mathbb{R}} \frac{\partial }{\partial v}G_{t-s}(x-v) \frac{1}{\psi_{\epsilon}^2(s,v)}\Big( \int_y^{\infty} G_{\epsilon}(v-z) \psi(s,z) dz\Big) 
\\ && \qquad   \times
 \Big( \psi_{\epsilon}(s,v) - \int_y^{\infty} G_{\epsilon}(v-z) \psi(s,z) dz \Big) \sigma_s(y)^2 dv dy ds.
\end{align*}
Adding and subtracting $\psi(s,v)$ in the second $dz$ integral, we get
\[
 H_\ep = \sum_{i=1}^3 H_{i,\ep},
 \]
 where
 \[
 H_{i,\ep}=
 \int_0^t \int_0^{\infty} \int_{\mathbb{R}} \frac{\partial }{\partial v}G_{t-s}(x-v) \frac{1}{\psi_{\epsilon}^2(s,v)}\Big( \int_y^{\infty} G_{\epsilon}(v-z) \psi(s,z) dz\Big)  F_i\sigma_s(y)^2 dv dy ds,
\]
where 
\begin{align*}
F_1 &:= \int_{-\infty}^y G_{\epsilon}(v-z) \psi_{\epsilon}(s,v)dz, \\
F_2 &:= \int_y^{\infty} G_{\epsilon}(v-z) \big[ \psi_{\epsilon}(s,v) - \psi(s,v) \big] dz, \\
F_3 &:= \int_y^{\infty} G_{\epsilon}(v-z) \big[ \psi(s,v)-\psi(s,z)\big] dz.
\end{align*}
We show convergence of each of these three terms, one at a time. To control the term  $H_{1,\ep}$, apply Minkowski's inequality, H\"older's inequality, and Lemma \ref{lem4.1},  to get, for anu $p\ge 2$
\begin{align*}
 \| H_{1,\ep} \|_p
&\leq C \int_0^t \int_0^{\infty} \int_{\mathbb{R}} \frac{\partial}{\partial v} G_{t-s}(x-v) \Big(\int_y^{\infty} G_{\epsilon}( v-z) dz \Big) \Big(\int_{-\infty}^y G_{\epsilon}( v-z) dz \Big) f(y)^2 dv dy ds.
\end{align*}
Notice that, for any fixed $s, y, v,$ we have 
$$
\frac{\partial }{\partial v}G_{t-s}(x-v) \Big(\int_y^{\infty} G_{\epsilon}( v-z) dz \Big) \Big(\int_{-\infty}^y G_{\epsilon}( v-z) dz \Big) f(y)^2 \to 0
$$
as $\epsilon \to 0$. Furthermore, we can trivially bound this integrand by
$$
 \Big|\frac{\partial }{\partial v}G_{t-s}(x-v)\Big| f(y)^2,
$$
which is independent of $\epsilon$, and $(dv \otimes dy \otimes ds)$-integrable on $\mathbb{R}\times (0,\infty) \times [0,t]$. Hence, by dominated convergence, $\|H_{1,\ep}\|_p\rightarrow 0$ as $\ep\rightarrow 0$.

We bound the term with $H_{2,\ep}$ as follows
\begin{align*}
\|H_{2,\ep}\|_p
&\leq C \sup_{s,v}\| \psi_{\epsilon}(s,v) - \psi(s,v) \|_{2p} \int_0^t \int_0^{\infty}\int_{\mathbb{R}} \Big|\frac{\partial }{\partial v}G_{t-s}(x-v)\Big| f(y)^2 dv dy ds
\end{align*}
for some positive constant $C>0$ and all $p\ge 2$. This quantity  converges to zero as $\epsilon \to 0$  by Lemma \ref{lem4.1}. Lastly, apply the same techniques to get
\begin{align*}
\| H_{3,\ep} \|_p
&\leq C \int_0^t \int_0^{\infty} \int_{\mathbb{R}} \Big|\frac{\partial }{\partial v}G_{t-s}(x-v)\Big| \Big(\int_y^{\infty} G_{\epsilon}(v-z) |v-z|^{1/2} dz \Big) f(y)^2 dv dy ds\\
&\leq C \epsilon^{1/4}.
\end{align*}
which converges to zero as $\epsilon \to 0$. Therefore, $ A_{4,\ep} + A_{5,\ep}$  converges to zero in $L^p(\Omega)$ as $\epsilon \to 0$, for all $p\ge 2$.
 
 \medskip
 \noindent
 {\it Step 5.} As a conclusion, we deduce that  the process $v(t,x)$ satisfies 
\begin{align*}
v(t,x) = \int_{\mathbb{R}} G_t (x-y) u_0(y) dy + \frac{1}{2} \int_0^t \int_{\mathbb{R}} \frac{\partial  }{\partial y} G_{t-s} (x-y) v(s,y)^2 dy ds \\+ \int_0^t \int_{\mathbb{R}} G_{t-s}(x-y) \sigma_s(y) W(ds,dy).
\end{align*} 
Since $u$ also satisfies this equation, we have $u \equiv v$ by uniqueness of solutions.
\end{proof}


\section{Regularity of Burgers' Equation}
We start with an easy, yet interesting, consequence of some of our results about $\psi$ and its regularity.
\begin{proposition}   \label{prop5.1}
Let $u(t,x)$ denote the solution to Burgers' equation (\ref{eq1}).  Assume that the initial condition $u_0$ is $\alpha$ H\"older continuous for some $\alpha \in (0,1)$. Then, for all $t,s\in [0,T]$, $x,y \in \mathbb{R}$, and $p \geq 2$, we have
$$
\| u(t,x) - u(s,y)\|_p \leq C(|t-s|^{\frac{\alpha}{2} \wedge (\frac 14-\frac 1{2q})}+ |x-y|^{\alpha\wedge (\frac 12-\frac 1q)}).
$$
\end{proposition}
\begin{proof}
Indeed, by adding and subtracting an appropriate term, we have
\begin{align*}
\| u(t,x) - u(t,y)\|_p &= 
2 \Big\| \frac{1}{\psi(s,y)}  \frac{\partial \psi}{\partial y}(s,y)-\frac{1}{\psi(s,x)}  \frac{\partial \psi}{\partial x}(s,x) \Big\|_p \\
&\leq 2 \Big\| \frac{\partial \psi}{\partial y}(s,y)\frac{\psi(s,x) -\psi(s,y)}{\psi(s,x)\psi(s,y)} \Big\|_p + 2 \Big\| \frac{1}{\psi(s,x)}\Big[\frac{\partial \psi}{\partial y}(s,y)-\frac{\partial \psi}{\partial x}(s,x)\Big]\Big\|_p \\
&\leq  C \left( |x-y|^{ 1/2 }+  |x-y|^{\alpha \wedge (\frac 12-\frac{1}{q})}  \right),
\end{align*}
where the last inequality follows from Cauchy-Schwarz inequality, (\ref{kk7}),  and Propositions \ref{prop3.4},  \ref{prop3.6}, and \ref{prop3.7}. Using the same technique of adding and subtracting appropriate terms yields the desired regularity in $t$.
\end{proof}

\begin{enumerate}
\item[(i)]  From Remark \ref{rem2} it follows that if we do not assume the H\"older continuity of $u_0$, then  we have
$$
\left\|   u(t,x) -  u(s,y)\right\|_p \leq C  (t\wedge s) ^{-1/2}  \big(|t-s|^{\frac 14 -\frac 1{2q}}+ |x-y|^{ \frac 12 -\frac 1q }\big),
$$
  Moreover, $ u(t,x)$ is continuous in $L^p(\Omega)$ on $[0,T]\times \RR$ for all $p\ge 2$.
  \item[(ii)] Proposition (\ref{prop5.1}) allows us to deduce the existence of a version of $u(t,x)$, which is locally H\"older continuous
  in space of order  $ \alpha \wedge (\frac 12-\frac 1{q})$  and in time of order  $ \frac{\alpha}{2} \wedge (\frac 14-\frac 1{2q})$.
  \end{enumerate}

The next proposition provides some moment estimates for the solution to Burgers equation.

\begin{proposition}
Let $u(t,x)$ denote the solution to Burgers' equation (1.1) Then, for all $t\in [0,T]$ and $x \in \mathbb{R}$, and $p \geq 2$, we have
\[
\sup_{x \in  \mathbb{R}}  \| u(t,x) \|_p  \le 
 K  c_{2p}  (t\vee 1)^{1- \frac 1{q}}  \exp\Big(   \| f\|_{L^2(\RR)}^2 t  \big(1/8+c_{2p} + \|f\|_{L^2(\mathbb{R})}^2 t/2 \big) +\frac{1}{4p}\|u_0\|_{L^1(\mathbb{R})} \Big),
\]
where $c_p$ is the optimal constant in Burkholder's inequality and $K$ is a constant depending on
 $q$, $\|f\| _{L^q(\RR)} $,  $\|f\| _{L^2(\RR)}$ , $ \|u_0\|_{L^1(\RR)}$, and  $\|u_0\|_\infty$.
\end{proposition}

\begin{proof}
By  H\"older's  inequality, we can write
\[
\|u(t,x)\|_p= 2  \left \| \psi(t,x)^{-1} \frac{\partial \psi}{\partial x}(t,x) \right\|_p 
\le 2 \left\| \psi(t,x)^{-1}  \right\|_{2p}  \left\|  \frac{\partial \psi}{\partial x}(t,x) \right\|_{2p}.
\]
Then, the result follows from Remark \ref{rem1} and Proposition  \ref{prop3.6}.
\end{proof}


\begin{thebibliography}{99}

\bibitem{burger1} L. Bertini, N. Cancrini, G. Jona-Lasinio, {\it The Stochastic Burgers Equation}, Commun. Math. Phys. {\bf 165}, 211-232 (1994)

\bibitem{chen} L. Chen and R. Dalang, \emph{Moments and Growth Indices for the Nonlinear Stochastic Heat Equation with Rough Initial Conditions}, Ann. Probab. {\bf 43}, 3006-3051 (2015)


\bibitem{burger2} I. Gy\"{o}ngy and D. Nualart, {\it On the Stochastic Burgers Equation in the Real Line}, Ann. Probab. {\bf 27}, 782-802 (1999)

\bibitem{burger3} J. Le\`{o}n, D. Nualart, and R. Pettersson, {\it The Stochastic Burgers Equation: Finite Moments and Smoothness of the Density}, Infinite Dimensional Analysis {\bf 3}, 363-385 (2000)

\bibitem{walsh} J. Walsh, {\it An Introduction to Stochastic Partial Differential Equations}, \`{E}cole d'\'{E}t\'{e} de Probabilit\'{e}s de Saint-Flour, XIV--1984, Lecture Notes in Math., {\bf 1180}, Springer, Berlin, 265-439 (1986)


\end{thebibliography}
\end{document}